\documentclass[10pt]{amsart}
\usepackage{latexsym, amsmath,amssymb}

\usepackage{hyperref}

\usepackage{xcolor}
\renewcommand{\epsilon}{\varepsilon}
\newcommand{\newsection}[1]
{\subsection{#1}\setcounter{theorem}{0} \setcounter{equation}{0}
\par\noindent}

\newtheorem{theorem}{Theorem}

\newtheorem{lemma}[theorem]{Lemma}
\newtheorem{corr}[theorem]{Corollary}

\newtheorem{proposition}[theorem]{Proposition}
\newtheorem{deff}[theorem]{Definition}

\newcommand{\bth}{\begin{theorem}}
\newcommand{\ble}{\begin{lemma}}
\newcommand{\bcor}{\begin{corr}}

\newcommand{\bdeff}{\begin{deff}}

\newcommand{\bprop}{\begin{proposition}}
\newcommand{\ele}{\end{lemma}}
\newcommand{\ecor}{\end{corr}}
\newcommand{\edeff}{\end{deff}}

\newcommand{\eprop}{\end{proposition}}

\newcommand{\cd}{\, \cdot\, }

\newcommand{\Rn}{{\mathbb R}^n}

\newcommand{\la}{\lambda}

\newcommand{\e}{\varepsilon}

\newcommand{\supp}{\text{supp }}
\renewcommand{\Pi}{\varPi}
\renewcommand{\Re}{\rm{Re} \,}

\renewcommand{\epsilon}{\varepsilon}

\newcommand{\dist}{{{\rm dist}}}

\newcommand{\R}{{\mathbb R}}

\newcommand{\1}{{\rm 1\hspace*{-0.4ex}%
\rule{0.1ex}{1.52ex}\hspace*{0.2ex}}}

\newcommand{\tb}{\widetilde \beta}

\begin{document}

\title%[unif]
{Uniform Sobolev estimates in $\mathbb{R}^{n}$ involving singular potentials}
%\thanks{The authors were supported in part by the NSF}
%
%
%
%
%
%

\thanks{The authors were supported in part by the NSF (NSF Grant DMS-1953413). }

\author[]{Xiaoqi Huang}
\address[X.H.]{Department of Mathematics,  Johns Hopkins University,
Baltimore, MD 21218}
\email{xhuang49@math.jhu.edu}

\author[]{Christopher D. Sogge}
\address[C.D.S.]{Department of Mathematics,  Johns Hopkins University,
Baltimore, MD 21218}
\email{sogge@jhu.edu}

\begin{abstract}
We generalize the Stein-Tomas~\cite{TomasRestriction} $L^2$-restricition theorem and the
uniform Sobolev estimates of Kenig, Ruiz and the second author~\cite{KRS} by allowing critically
singular potential.  We also obtain Strichartz estimates for Schr\"odinger and wave operators with
such potentials.  Due to the fact that there may be nontrivial eigenfunctions we are required to make
certain spectral assumptions, such as assuming that the solutions only involve sufficiently large frequencies.
 \end{abstract}

\maketitle
\setcounter{secnumdepth}{3}
   
\newsection{Introduction and main results}

The main purpose of this paper is to extend the uniform Sobolev inequalities in $\mathbb{R}^n$ of Kenig, Ruiz and the second author \cite{KRS}, as well as the $L^2$-restricition theorem of Stein and Tomas~\cite{TomasRestriction} to include Schr\"odinger operators,
\begin{equation}\label{def0}
H_V=-\Delta+V(x)
\end{equation}
with critically singular potentials $V$, which are assumed to be real-valued and 
\begin{equation}\label{def}
V\in L^{n/2}(\mathbb{R}^n).
\end{equation}

In a recent work of the authors with Blair and Sire \cite{blair2020uniform}, the same problem was considered for compact Riemannian manifolds, which generalizes results in
%Dos Santos Ferreira, Kenig and Salo
\cite{DKS}, \cite{BSSY} and \cite{Shya}. Also, in an earlier work of Blair, Sire and the second author \cite{BSS}, quasimode and related spectral projection estimates were discussed in
 $\R^n$ under the additional assumption that $V\in \mathcal{K}$, the Kato class. The spaces $L^{n/2}$ and ${\mathcal K}$ have the same
scaling properties, and both obey the scaling law of the Laplacian, which accounts for their criticality. 
It is natural to assume that $V\in {\mathcal K}$ when dealing with large exponents, and, as was shown in \cite{BSS}, this assumption is
needed to obtain optimal sup-norm estimates for quasimodes.
%It is necessary to assume that $V\in {\mathcal K}$ to obtain optimal sup-norm estimates.
However, motivated by the results in \cite{blair2020uniform}, for the smaller exponents arising in uniform Sobolev inequalities, it is natural to merely assume that $V$ in $L^{n/2}$, which we shall discuss below in Theorem~\ref{unifSob}.

As was shown in the appendix of \cite{blair2020uniform}, if $V\in L^{n/2}$ then $H_V$ is essentially self-adjoint and bounded from below. As a result, we shall assume throughout that 
\begin{equation}\label{below}
\text{Spec } H_V\subset [-N_0, +\infty), \,\,\text{i.e.},\,\, H_V+N_0\ge 0,
\end{equation}
for some fixed positive number $N_0$ which depends on $V$.

The uniform Sobolev estimates and quasimode estimates that we can obtain are the following.
\begin{theorem}\label{unifSob}
Let $n\ge 3$ and suppose that
\begin{equation}\label{1.3}
\min\bigl(q, \, p(q)'\bigr)> \tfrac{2n}{n-1},
\quad \text{and } \, \tfrac1{p(q)}-\tfrac1q=\tfrac2n.
\end{equation}
 Then
if $V\in L^{n/2}(\mathbb{R}^n)$ is real-valued and $\Lambda$, $\delta>0$ are fixed constants with $\Lambda=\Lambda(q, n, V)$ sufficiently large, we have the uniform bounds
\begin{equation}\label{1.4}
\|\bigl(H_V-\zeta\bigr)^{-1}u\|_q\lesssim \bigl\|u\bigr\|_{p(q)}, \quad \text{if } \, 
\, \, \zeta\in \Omega_\delta,
\end{equation}
where
\begin{equation}\label{1.5}
 \Omega_\delta = \{\zeta \in {\mathbb C}\,\setminus\,[-N_0,+\infty): \,\,
\mathrm{dist }(\zeta, [-N_0,+\infty))
 \ge \delta
\, \, 
\text{if } \, \, \Re \zeta < \text{$\Lambda^2$} \}.
%(\Im\zeta)^2\ge 1+ c_0 |\Re \zeta|
\end{equation} 
%assuming that $|\zeta|\ge 1$.
Also, suppose that 
\begin{equation}\label{1.6}
\tfrac{2(n+1)}{n-1}\le q\leq \tfrac{2n}{n-4}, \,\,\,\text{if}\,\,\, n\ge 5,\,\,\,\text{or}\,\,\,\tfrac{2(n+1)}{n-1}\le q< \infty,\,\,\ \text{if}\,\,\, n=3, 4.
\end{equation}
Then if $u\in\text{Dom}(H_V)$, for any $0<\e<\la/2$, we have
\begin{equation}\label{qm}
\|u\|_q\lesssim \la^{n(1/2-1/q)-3/2}\e^{-1/2}\|(H_V-\la^2+i\la\e) u\|_2,
\, \, \, \text{if } \, \la\ge \Lambda.
\end{equation}
\end{theorem}
Here, $\text{Dom}(H_V)$ denotes the domain of $H_V$ and $N_0$ is as in \eqref{below}.   Also, $r'$ denotes the conjugate exponent for $r$, i.e.,
the one satisfying $1/r+1/r'=1$.  
Additionally, we are using the notation that $A\lesssim B$ means that $A$ is bounded from above by a constant times $B$.  The implicit constant might depend on the parameters involved, such as $n$, $q$ and $V$, but not on $\zeta$, $\la$ or $\e$ in \eqref{1.4} and \eqref{qm}.

The condition \eqref{1.3} on the range of exponents was shown to be
be sharp in \cite{KRS}.
The gap condition in \eqref{1.3} that $ \tfrac1{p(q)}-\tfrac1q=\tfrac2n$, follows from scaling considerations,
while the necessity of the first part of \eqref{1.3}  is related to the fact that the Fourier transform of surface measure on the
sphere in $\Rn$ is not in $L^q(\Rn)$ if $q\le \tfrac{2n}{n-1}$. The condition \eqref{1.6} on the exponents is also sharp, since it agrees with conditions in the standard quasimode estimates when $V\equiv 0$ except for the case $n=3,\, q=\infty$. In that case, \eqref{qm} may not be valid if we only assume $V\in L^{n/2}(\R^n)$ due to the possible existence of unbounded eigenfunctions for the operator $H_V$.  See e.g., \cite{BSS} for more details.

\eqref{1.4}--\eqref{1.5} imply that we have uniform $L^{p(q)}\to L^q$ operator bounds for
$(H_V-\zeta)^{-1}$ if $\text{Re }\zeta$ is large and $\text{Im }\zeta\ne0$, which is a natural analog of the uniform Sobolev
estimates of Kenig, Ruiz and the second author \cite{KRS}, and, in the special case where $V\equiv0$ is equivalent
to the results in \cite{KRS}. Inequalities of this type, as well as weighted $L^2$-estimates for $(H_V-\zeta\bigr)^{-1}$, have been extensively studied for different types of potentials; see, e.g, \cite{jensen1979spectral}, \cite{rodnianski2015effective}, \cite{bouclet2018uniform}, \cite{mizutani2020uniform}. In particular, in \cite{mizutani2020uniform}, it is proved by Mizutani that \eqref{1.4}--\eqref{1.5} hold for $V\in L^{n, \infty}_0 (\R^n)$, where $L^{n, \infty}_0$ denotes the completion of $C_0^\infty$ functions under the $ L^{n, \infty}_0$ norm. Although $L^{n/2}	\hookrightarrow  L^{n, \infty}_0$  , the proof of \eqref{1.5} is based on a different method and we shall discuss at the end of section 2 that how we can modify the proof there to further weaken the conditions on $V$.

By results in \cite{SoggeZelditchQMNote} the bounds in 
\eqref{qm} are equivalent to the following spectral projection bounds
\begin{equation}\label{1.9}
\|\chi^V_{[\la, \la+\e]}\|_{L^2(\R^n)\to L^q(\R^n)}
\lesssim \e^{1/2}\la^{n(1/2-1/q)-1/2}, \quad \forall \,\,0<\e<1,\,\,
\la \ge \Lambda,
\end{equation}
for some $\Lambda$ large enough and $q$ as in \eqref{1.6}, if 
$\chi^V_{[\la, \la+\e]}$ denotes the spectral projection operator which projects onto the part of the spectrum of $H_V$
%\footnote{Here we are assuming, as we may after adding a constant to $V$ if necessary, that $H_V$ is nonnegative.}
 in the corresponding shrinking intervals 
$[\la^2, (\la+\e)^2]$.
%  Although, the results in
%the preceding theorem are all that we need to prove 
%Theorem~\ref{nonpossob}, we shall show that we
If in addition to \eqref{def} we assume that $V$
is in the Kato class then we 
also have \eqref{1.9}, as in the
$V\equiv 0$ case 
for all $p\ge\tfrac{2(n+1)}{n-1}$, which is equivalent to the Stein-Thomas restriction theorem for the sphere (see, i.e., Chapter 5 in \cite{SFIO2}).

In \cite{BSS}, analogs of \eqref{1.9} were obtained for all $\la\ge0$ under the assumption that the $L^{n/2}$ norm of $V$ is small. In that case, using a simple argument involving Sobolev estimates, it is not hard 
to show that $H_V=-\Delta+V$ itself defines a positive self-adjoint operator which can not have negative spectrum. We generalize the results in \cite{BSS} by removing the smallness condition at the cost 
of ignoring lower part of the spectrum.

By letting $\e\rightarrow 0$, inequalities like \eqref{1.9} certainly implies the absence of embedded eigenvalues in the corresponding region of spectrum. As far as the problem of absence of 
eigenvalues is considered, the $L^{n/2}$ norm here is critical, since in \cite{koch2002}, Koch and Tataru showed that if $q<n/2$ there are
examples of compactly supported potentials $V\in L^q$ and smooth compact supported functions $u$ such that $H_Vu=0$. In the opposite direction, Ionescu and Jerison \cite{2003absence} proved that $H_V$ does not admit positive eigenvalues for $V\in L^{n/2}_{loc}$ with certain decay conditions as $|x|\rightarrow \infty$.

Since $H_V$ does not admit large eigenvalues, if $E^\prime(\la)$ denotes the density of the spectral measure associated to the operator $H_V$, with $\int E^\prime(\la)d\la$ being the resolution of identity, by Stone's formula,
$$E^\prime(\la)=\frac{1}{2\pi i}\lim_{\e\rightarrow 0} \big((H_V-\la-i\e)^{-1}-(H_V-\la+i\e)^{-1}\big), \,\,\,\text{if}\,\,\,\la>\Lambda^2,
$$
which is equivalent to 
$$E^\prime(\la)=\frac{1}{\pi}\lim_{\e\rightarrow 0} \,\e\cdot(H_V-\la-i\e)^{-1}(H_V-\la+i\e)^{-1}.
$$
Thus, for $\la, \,\e$ defined as above, by applying \eqref{qm} with $\la^\prime=\la^{1/2}$, $\e^\prime=\e/\la^{1/2}$ and duality, we have the following restriction type estimates
\begin{corr}
Suppose that $q$ satisfies \eqref{1.6}, then we have 
\begin{equation}
\|E^\prime(\la) f\|_{L^q(\R^n)}\le C\la^{\frac{n}{2}(\frac{1}{q^\prime}-\frac{1}{q})-1} \|f\|_{L^{q^\prime}(\R^n)}\,\,\,\text{if}\,\,\,\la>\Lambda^2.
\end{equation}
\end{corr}

Among other things, if $\chi^V_{(-\infty, \Lambda]}$ denotes the spectral projection onto the interval $(-\infty, \Lambda^2)$ for $H_V$, for $q$ satisfying \eqref{1.6}, the quasimode estimates \eqref{qm} also implies
\begin{equation}\label{1.10}
\|\chi^V_{(-\infty, \Lambda]}\|_{L^2(\R^n)\to L^q(\R^n)} \lesssim 1,
\end{equation}
since $H_V$ is bounded from below.
Using \eqref{1.9} and \eqref{1.10}, it is straightforward to adapt the proof of Theorem 8.1 and Theorem 9.5 in \cite{BSS} to obtain the following
\begin{theorem}\label{strthm}  Let $n\ge3$ and fix
$V\in L^{n/2}(\R^n)$.  Let $u$ be the solution of
\begin{equation}\label{1.11}
\begin{cases}
\bigl(\partial_t^2-\Delta+V(x)\bigr)u=0
\\
u|_{t=0}=f_0, \quad \partial_tu|_{t=0}=f_1.
\end{cases}
\end{equation}
Then for $p_c=\frac{2(n+1)}{n-1}$, we have
\begin{equation}\label{1.12}
\|u\|_{L^{p_c}([0,1]\times \R^n)}\le C_V\bigl(\|(i+H_V)^{1/4}f_0\|_{L^2(\R^n)}
+\|(i+H_V)^{-1/4}f_1\|_{L^2(\R^n)}\bigr).
\end{equation}
Additionally, suppose that $\chi\in C^\infty(\R)$ is a smooth function satisfying
\begin{equation}\label{high1}
\chi(\la)=1\,\,\,\text{for}\,\, \,\la\ge 1\,\,\,\text{and}\,\,\,\chi(\la)=0\,\,\,\text{for}\,\, \,\la\le 1/2.
\end{equation}
Then 
\begin{equation}\label{1.13}
\|\chi(H_V/M)u\|_{L^{p_c}(\R\times \R^n)}\le C_V\bigl(\|(i+H_V)^{1/4}f_0\|_{L^2(\R^n)}
+\|(i+H_V)^{-1/4}f_1\|_{L^2(\R^n)}\bigr),
\end{equation}
assuming that $M$ is a large enough constant which depends on $V$.
\end{theorem}

Here the operator $\chi(H_V/M)$ denotes a smooth spectral projection onto the interval $(M/2, +\infty)$ for $H_V$. More specifically, 
\begin{equation}\label{def2}
\chi(H_V/M)f=\int \chi(\la/M)E^\prime(\la)fd\la.
\end{equation}

Similarly, as \eqref{qm}, \eqref{1.13} suggest, by projecting onto the subspace with large spectrum, we have the following global Strichartz estimate for Schr\"odinger equations. 
\begin{theorem}\label{hvthm}  
Let the potential $V\in L^{n/2}(\mathbb{R}^n)$ be real-valued, and $\chi\in C^\infty(\R)$ be as in \eqref{high1}. Then for each pair of exponents $(p,q)$ satisfying 
\begin{equation}\label{i.2}
n(1/2-1/q)=2/p, \, \,  2\le p\le \infty\,  \, \text{and}\,\,\,n\ge 3,
\end{equation}
we have 
\begin{equation}\label{strichartz}
\bigl\|\chi(H_V/M)e^{-itH_V}f\bigr\|_{L^p_tL^q_x(\R\times \R^n)}
\lesssim \|f\|_{L^2(\mathbb{R}^n)},
\end{equation}
where $M$ is a large constant which depends on $V$. Additionally, if $\,V\in L^{n/2}(\mathbb{R}^n)+L^{\infty}(\mathbb{R}^n)$ is real-valued, then for $(p,q)$ satisfying \eqref{i.2}, we have
\begin{equation}\label{localstrichartz}
\bigl\|e^{-itH_V}f\bigr\|_{L^p_tL^q_x([0,1]\times \R^n)}
\lesssim \|f\|_{L^2(\mathbb{R}^n)}.
\end{equation}
\end{theorem}

If $V\equiv 0$, it is well-known that the Strichartz estimates \eqref{strichartz} are a consequence of the following dispersive estimate
\begin{equation}\label{disp}
\bigl\|e^{it\Delta}f\bigr\|_{L^\infty(\R^n)}
\lesssim t^{-n/2}\|f\|_{L^1(\mathbb{R}^n)},
\end{equation}
with the endpoint of above estimates corresponding to $p=2, q=\frac{2n}{n-2},\,n\ge 3$ obtained by Keel and Tao in \cite{KT}. However, the natural $L^1\rightarrow L^\infty$ dispersive estimates for the 
operator $e^{-itH_V}$, as well as the Strichartz estimates \eqref{strichartz} may break down due to the possible existence of bounded states. If $V\in L^{n/2}(\R^n)$, it was proved by Goldberg in  \cite{goldberg2009strichartz} that global Strichartz estimates \eqref{strichartz} hold with $f$ projected onto the continuous part of the spectrum under the assumption that $0$ is neither a eigenvalue or resonance. Under the same condition, Mizutani \cite{mizutani2020uniform} also proved similar inhomogeneous Strichartz estimates when the pairs $(p,q)$ are outside the admissible range in \eqref{i.2}. The inequalities in \eqref{strichartz} show that, even if $0$ is a eigenvalue or resonance, the same estimates still hold as long as $f$ is projected onto the higher part of the spectrum. See also 
\cite{jss} for related high energy estimates for a different class of potentials $V$.

We are grateful to Mizutani~\cite{mizutani2020uniform} for pointing out his recent work after a preliminary version of this paper was
completed.

%\newsection{Strichartz estimates for Schr\"odinger operators with singular potentials in Euclidean space}

\newsection{Uniform Sobolev inequalities and quasimode estimates in $\mathbb{R}^n$}\label{unif}

In this section we shall prove Theorem \ref{unifSob}. As in \cite{blair2020uniform}, the main idea is to use the second resolvent formula 
\begin{equation}\label{2.1}
(-\Delta+V-\zeta)^{-1}-(-\Delta-\zeta)^{-1} = -(-\Delta-\zeta)^{-1}\, V(\, -\Delta+V-\zeta)^{-1}, \quad \mathrm{Im} \, \zeta \ne 0,
\end{equation}
%along with quasimode estimates for $H_V$ from \cite{BSS} and \cite{sogge88},
%as well as 
along with quasimode estimates and 
 uniform Sobolev estimates for the 
unperturbed operator $H_0=-\Delta$ from \cite{KRS} and \cite{BSS}. Specifically, we shall require that for $n\ge 3$
\begin{equation}\label{2.2}
\|(-\Delta-\zeta)^{-1}f\|_{L^q(\R^n)} \le C\|f\|_{L^p(\R^n)},\,\,\, \forall \,  \zeta\in \mathbb{C}\setminus[0,+\infty),
\end{equation}
for pairs of exponents $(p,q)$ satisfying \eqref{1.3}, which is due to Kenig, Ruiz, and the second author \cite{KRS}.  We  also need the quasimode estimates
\begin{equation}\label{2.3}
\|u\|_{L^{p_c}(\R^n)} \le C\la^{-1+1/p_c}\e^{-1/2}\|(-\Delta-\la^2+i\e\la)f\|_{L^2(\R^n)},\,\,\, \forall \,\,\, 0<\e<\la/2,
\end{equation}
where $p_c=\tfrac{2(n+1)}{n-1}$. By a change of scale argument, it is not hard to check that \eqref{2.3} is an equivalent version of Stein-Thomas restriction theorem for $\R^n$ (see, e.g., \cite{BSS} Proposition 9.3). Actually, similar estimates also hold in the case $\e\ge \la/2$, 
but we skip these here since they are less useful.
%we skip here since it is less useful in applications.

To prove Theorem \ref{unifSob}, a key ingredient is the following theorem.
\begin{theorem}\label{lemm}
Let $n\ge 3$ and $\tfrac{2n}{n-1}<q<\tfrac{2n}{n-3}$, then if $\,V\in L^{n/2}(\mathbb{R}^n)$ is fixed, we have 
\begin{equation}\label{2.4}
\|(-\Delta-\la^2+i\e\la)^{-1}\, (V f)\|_{q}\le 1/2 \|f\|_q, \,\,\,\forall \,\,\la\ge\Lambda, \,\,\, \e>0,
\end{equation}
assuming that
$\Lambda=\Lambda(q,n,V)\ge 1$ sufficiently large.
\end{theorem}

\eqref{2.4} essentially follows from Lemma 3.3 in \cite{mizutani2020uniform}, where the author gave a short proof of this lemma using uniform Sobolev estimates for the free resolvent operator when $\frac{2}{n+2}\le\frac1p-\frac1q\le \frac2n$. The proof of \eqref{2.4} that we shall give is more complicated since it requires a decomposition of the free resolvent operator into dyadically localized operators.  However, we obtain
results that are
 less restrictive on the conditions required for $V$.  See the discussion at the end of this section for more details.

We shall postpone the proof of Theorem~\ref{lemm} to the end of this section and first see how we can apply the above theorem to obtain \eqref{1.4} and \eqref{qm}.  Even though the proof mostly follows from the arguments in \cite{blair2020uniform}, we include it here for the sake of completeness.

To prove \eqref{1.4}, if $\text{Re} \,\zeta \ge\Lambda^2$, it is equivalent to showing that 

$$\bigl\| (H_V-\la^2+i\e\la)^{-1}\bigr\|_{L^p\to
L^q}\lesssim 1, \quad \text{if } \, \,
\la \ge \Lambda, $$
with $\Lambda$ sufficiently large and $(p,q)$ as
in \eqref{1.3}.  By duality, it suffices prove this inequality when 
\begin{equation}\label{2.5}
\tfrac{2n}{n-1}<q\le \tfrac{2n}{n-2}.
\end{equation}
Thus, our task is to show that
\begin{equation}\label{2.6}
\bigl\| (H_V-\la^2+i\e\la)^{-1}f\bigr\|_{L^q(\R^n)}
\le C\|f\|_{L^p(\R^n)} \quad \text{if } \, \,
\la \ge \Lambda,
\end{equation}
with $(p,q)$ satisfying \eqref{1.3} and \eqref{2.5}.
We are also assuming that \eqref{2.2} and \eqref{2.3} are valid for this pair of exponents.

We are assuming \eqref{2.5} since by Sobolev estimates (see, e.g., (6.5) in the appendix of \cite{blair2020uniform}),
we have
$$u\in L^q(\R^n), \quad 2\le q\le \tfrac{2n}{n-2} \, 
\, \, \text{if } \, \, \,
(H_V-\la^2+i\e\la)u\in L^2.$$
%{\color{red}
%{\bf $\clubsuit$ put this in the appendix so that we can just quote this and not
%require reader to think about Sobolev. $\clubsuit$}}
Thus for $q$ as in \eqref{2.5}
\begin{equation}\label{2.7}
\bigl\| (H_V-\la^2+i\e\la)^{-1}f\bigr\|_{L^q(\R^n)}
<\infty \quad \text{if } \, \, f\in L^2(\R^n).
\end{equation}
In proving \eqref{2.6}, since $L^2$ is dense in $L^p$
we may and shall assume that $f\in L^2(\R^n)$ to be
able to use \eqref{2.7} to justify a bootstrapping
argument that follows.

By using the second resolvent formula \eqref{2.1}, write
\begin{align}\label{2.8}
(H_V-\la^2+i&\e\la)^{-1}f
\\
= &(-\Delta-\la^2+i \e\la)^{-1}f
\notag
\\
&- (-\Delta-\la^2+i \e\la)^{-1}
\bigl(V\cdot (H_V-\la^2+i \e\la)^{-1}f\bigr)
\notag
\\
&=I-II. \notag
\end{align}

If we apply the uniform Sobolev estimates \eqref{2.2} for the 
unperturbed operator, we have
\begin{equation}\label{2.9}
\|I\|_q\le C\|f\|_p,
\end{equation}
while by using \eqref{2.4} in Theorem~\ref{lemm}, 
\begin{equation}\label{2.10}\|II\|_q \le 1/2
\bigl\| (H_V-\la^2+i \e\la)^{-1}f\bigr\|_{L^q},\,\,\,\text{if}\,\,\, \la\ge\Lambda.
\end{equation}
Thus, if we combine \eqref{2.8}, \eqref{2.9} and \eqref{2.10}, we conclude that for $\la\ge\Lambda$ we have
$$\|(H_V-\la^2+i \e\la)^{-1}f\|_{L^q(\R^n)}
\le C\|f\|_{L^p(\R^n)}
+\frac12 \, \| (H_V-\la^2+i\e\la)^{-1}f\bigr\|_{L^q(\R^n)}.$$
By \eqref{2.7}, this leads to \eqref{2.6} since
we are assuming, as we may, that $f\in L^2(\R^n)$.

Now we shall give the proof of quasimode estimates \eqref{qm}, which are needed later in the proof of \eqref{1.4} for $\text{Re} \,\zeta <\Lambda^2$.  First, note that by the quasimode estimates \eqref{2.3} for the unperturbed operator,
\begin{equation}\label{2.11}
\|I\|_{p_c}\le C \la^{-1+1/p_c}\e^{-1/2}\|f\|_2, \,\,\text{if} \,\,0<\e<\la/2.
\end{equation}
If we combine \eqref{2.8}, \eqref{2.10} and \eqref{2.11}, we conclude that for $\la\ge\Lambda$ we have
$$\|(H_V-\la^2+i \e\la)^{-1}f\|_{L^{p_c}(\R^n)}
\le C \la^{-1+1/p_c}\e^{-1/2}\|f\|_{L^{2}(\R^n)}
+\frac12 \, \| (H_V-\la^2+i\e\la)^{-1}f\bigr\|_{L^{p_c}(\R^n)}.$$
By \eqref{2.7}, this leads to \eqref{qm} for $q=p_c=\tfrac{2(n+1)}{n-1}$, since in this case $-1+1/p_c=n(1/2-1/p_c)-3/2$.

The remaining estimates in \eqref{qm} for exponents $q>\frac{2(n+1)}{n-1}$ as in \eqref{1.6} now just 
are a consequence of 
the following theorem.
\begin{theorem}\label{big} 
Let $n\ge5$, assume that \eqref{qm} holds for some $\tfrac{2(n+1)}{n-1}\leq r< \tfrac{2n}{n-4}$ and $\la\ge \Lambda(r, V)$, with $0<\e<\la/2$.
Then if $V\in L^{n/2}(\R^n)$ we have for $u\in \text{Dom}(H_V)$
\begin{equation}\label{2.12}
\|u\|_{q}\le C  \,
\la^{n(1/2-1/q)-3/2} \, \e^{-1/2}\bigl\|(-\Delta+V-\la^2+i \e\la)u\bigr\|_2, 
 \, \, \text{if } \, 
\la \ge \Lambda,\,\,\,\, r< q\leq \tfrac{2n}{n-4}.
\end{equation}
Similarly, for $n=3$ or $n=4$, assume that \eqref{qm} holds for some $\tfrac{2(n+1)}{n-1}\leq r< \infty$, with $0<\e<\la/2$, then we have
\begin{equation}\label{2.13}
\|u\|_{q}\le C\,
\la^{n(1/2-1/q)-3/2} \,\e^{-1/2}\bigl\|(-\Delta+V-\la^2+i\e\la)u\bigr\|_2, \\
 \, \, \, \text{if } \, \, 
\la \ge \Lambda,\,\,\, r< q<\infty,
\end{equation}
assuming that $\Lambda=\Lambda(q, n, V)$ in \eqref{2.12} and \eqref{2.13} are sufficiently large.
\end{theorem}

Theorem~\ref{big} is essentially the analog of Theorem 2.3 in \cite{blair2020uniform}, which says we can use the quasimode estimates for smaller exponents $r$ to obtain quasimode estimates for larger exponents $q$ up to the optimal range. Here compared with Theorem 2.3 in \cite{blair2020uniform}, we do not have to assume that $\e$ has a lower bound which depends on $\la$ by requiring that $\la\ge \Lambda(q, n, V)$, as in the case of uniform Sobolev inequalities, \eqref{1.4}.

As in \cite{blair2020uniform}, the proof of Theorem~\ref{big} requires the following lemma.
\begin{lemma}\label{Sobv}
Let $V\in L^{n/2}(\R^n)$ be real valued, then there exists a constant $N_0>1$ large enough such that for $u\in \text{Dom}(H_V)$, we have
\begin{multline}\label{2.14}
\|u\|_{q}\le \bigl\|(-\Delta_g+V+N_0)u\bigr\|_2,\\ \,\text{for}\,\,\, 2<q< \infty, \,\,\text{if} \,\, n=3, 4,\,\,\,\text{or}\,\,\, 2<q\le \tfrac{2n}{n-4}\,\,\text{if}\,\,n\ge 5.
\end{multline}
\end{lemma}

Lemma~\ref{Sobv} is essentially a special case of Lemma 2.4 in \cite{blair2020uniform}, where the authors proved a sharp Sobolev type estimates for the operator $H_V$ on compact manifolds. However, their proof works equally well in the case of $\R^n$. Thus, for 
the sake of brevity, we shall skip the proof of \eqref{2.14} here and refer the reader to
the proof of Lemma 2.4
%to page 14 
in \cite{blair2020uniform} for details. The main idea is that although the uniform resolvent estimates for the operator $(-\Delta-\zeta)^{-1}$ hold for a rather restricted range of exponents 
$(p,q)$, it is possible to enlarge the range of $(p,q)$ at the expense of requiring $\zeta$ to lie in certain regions of the complex plane.

\begin{proof}[Proof of Theorem \ref{big}]
Throughout the proof we shall assume that 
\begin{equation}\label{2.61}
\tfrac{2(n+1)}{n-1}\le r<q\leq \tfrac{2n}{n-4}, \,\,\,\text{if}\,\,\, n\ge 5,\,\,\,\text{or}\,\,\,\tfrac{2(n+1)}{n-1}\le r<q< \infty,\,\,\ \text{if}\,\,\, n=3, 4.
\end{equation}
Note that proving \eqref{2.12} and \eqref{2.13} is equivalent to showing that for $q$ satisfying \eqref{2.61}
\begin{equation}\label{2.62}
\bigl\|(H_V-\la^2+i \e\la)^{-1}f\bigr\|_q \leq C \la^{n(1/2-1/q)-3/2} \,\e^{-1/2} \|f\|_2,
 \, \, \, \text{if } \, \, 
\la \ge \Lambda.
\end{equation}
The proof of \eqref{2.12} and \eqref{2.13} also relies on the simple fact that if we let

\begin{equation}\label{2.18}
V_{\le N}(x)=
\begin{cases}V(x), \, \, \, \text{if } \, \,
|V(x)|\le N,
\\
0, \, \, \, \text{otherwise},
\end{cases}
\end{equation}
then, of course,
\begin{equation}\label{2.19}
\|V_{\le N}\|_{L^\infty}\le N,
\end{equation}
and, if $V_{>N}(x)=V(x)-V_{\le N}(x)$,
\begin{equation}\label{2.20}
\|V_{>N}\|_{L^{n/2}(\R^n)}\le \delta(N), 
\quad \text{with } \, \, \delta(N) \searrow
0, \, \, \, \text{as } \, \, N\to \infty,
\end{equation}
since we are assuming that $V\in L^{n/2}(\R^n)$.

Fix a smooth bump function $\beta\in C_0^\infty(1/4,4)$ with $\beta\equiv 1$ in $(1/2,2)$, and let $P=\sqrt{-\Delta}$, write
\begin{multline}\label{2.64}
(H_V-\la^2+i\e\la)^{-1}f
\\ \qquad \qquad\qquad= \beta(P/\la)(H_V-\la^2+i\e\la)^{-1}f+ \big(1-\beta(P/\la)\big)(H_V-\la^2+i\e\la)^{-1}f \\
=A+B. \qquad\qquad \qquad\qquad\qquad \qquad\qquad\qquad \qquad\qquad\qquad \qquad\qquad \quad
\end{multline}

To deal with the first term, note that since $m(\xi)=\la^{-\alpha}|\xi|^\alpha \beta(|\xi|/\la)$ is a Mikhlin multiplier, i.e., $$\big||\xi|^k\nabla^km(\xi)\big|\le C_k,\,\,\forall\,\, k\ge 0,\,\,\xi\neq 0.$$ 
Therefore, by the Mikhlin multiplier theorem(see, e.g., \cite{SFIO2} Theorem 0.2.6), we have
\begin{equation}\label{2.65}
\|(-\Delta)^{\frac{\alpha}{2}} \beta(P/\la)\|_{L^r\rightarrow L^r} \lesssim \la^{\alpha},\,\,\,\,\, \text{if}\,\,\, 1<r<\infty.
\end{equation} So by Sobolev estimates, \eqref{2.65} and \eqref{qm} for the exponent $r$, if $\alpha=n(\frac1r-\frac1q)$ and $0<\e<\la/2$, we have
\begin{equation}\label{2.66}
\begin{aligned}
\|A\|_q &\leq \|(\Delta_g)^{\frac{\alpha}{2}}\beta(P/\la)(H_V-\la^2+i\e\la)^{-1}f\|_r  \\
&\le \la^{n(\frac1r-\frac1q)} \|(H_V-\la^2+i\e\la)^{-1}f\|_r \\
&\le  C \,\la^{n(1/2-1/q)-3/2} \,\e^{-1/2}\|f\|_2,\,\,\,\text{if}\,\,\,\la\ge \Lambda_1=\Lambda(r).
\end{aligned}
\end{equation}

To bound the second term, we shall use the second resolvent formula \eqref{2.1} to write
\begin{equation}
\begin{aligned}
\big(1-\beta(P/\la)\big)(&H_V-\la^2+i \e\la)^{-1}f
\\
= &\big(1-\beta(P/\la)\big)(-\Delta-\la^2+i \e\la)^{-1}f
\notag
\\
&- \big(1-\beta(P/\la)\big)(-\Delta-\la^2+i \e\la)^{-1}
\bigl(V_{>N}\cdot (H_V-\la^2+i \e\la)^{-1}f\bigr)
\notag
\\
&- \big(1-\beta(P/\la)\big)(-\Delta-\la^2+i \e\la)^{-1}
\bigl(V_{\le N}\cdot (H_V-\la^2+i \e\la)^{-1}f\bigr)
\notag
\\
&=I-II-III. \notag
\end{aligned}
\end{equation}

Since the function $1-\beta(\tau/\la)$ vanishes in a dyadic neighborhood of $\la$, it is easy to see that 
$$\big(1-\beta(|\xi|/\la)\big)(|\xi|^2-\la^2+i \e\la)^{-1}(|\xi|^2+1)
$$
is a Mikhlin multiplier, and so
\begin{equation}\label{2.67}
\|\big(1-\beta(P/\la)\big)(-\Delta-\la^2+i \e\la)^{-1}f\|_q\le \|(-\Delta+1)^{-1}f\|_q,\,\,\,\,\, \text{if}\,\,\, 1<q<\infty.
\end{equation}
So by \eqref{2.67}, Sobolev estimates, and H\"older's inequality, we have for $q$ satisfying \eqref{2.61}
\begin{equation}\label{2.68}
\begin{aligned}
\|II\|_q&\le \|(-\Delta_g+1)^{-1}\big(V_{> N}\cdot (H_V-\la^2+i \e\la)^{-1}\big)f\|_q  \\
&\le \|V_{> N}\cdot (H_V-\la^2+i \e\la)^{-1}f\|_{p(q)} \\
&\le  C\delta(N)\|(H_V-\la^2+i \e\la)^{-1}f\|_q,
\end{aligned}
\end{equation}
where $\frac{1}{p(q)}-\frac1q=\frac2n$. By \eqref{2.20} we can
fix $N$ large enough so that
$C\delta(N)<1/4$, yielding the bounds
\begin{equation}\label{2.69}
\|II\|_q<\frac14 \, \bigl\|
(H_V-\la^2+i\e\la)^{-1}f
\bigr\|_q.
\end{equation}

For the third term $III$, note that as before, by the support property of $\beta$, it is straightforward to check that
$$\big(1-\beta(|\xi|/\la)\big)(|\xi|^2-\la^2+i \e\la)^{-1}\la^2
$$
is a Mikhlin multiplier, and so
\begin{equation}\label{2.70}
\|\big(1-\beta(P/\la)\big)(-\Delta-\la^2+i \e\la)^{-1}f\|_q\le C\la^{-2}\|f\|_q,\,\,\,\,\, \text{if}\,\,\, 1<q<\infty.
\end{equation}
Thus, by \eqref{2.19}, we have
\begin{equation}\label{2.71}
\begin{aligned}
\|III\|_q&\le C\la^{-2}\|\big(V_{\le N}\cdot (H_V-\la^2+i \e\la)^{-1}\big)f\|_q  \\
&\le CN\la^{-2}\|(H_V-\la^2+i \e\la)^{-1}f\|_q.
\end{aligned}
\end{equation}
If we choose $\Lambda_2$ such that $\Lambda_2^2=4CN$, we have
\begin{equation}\label{2.72}
\|III\|_q<\frac14 \, \bigl\|
(H_V-\la^2+i\e\la)^{-1}f
\bigr\|_q,\,\,\,\text{if}\,\,\,\la\ge\Lambda_2.
\end{equation}

We are left with estimating the first term $I$.  Note that for $q$ satisfying \eqref{2.61}, we have $\frac12-\frac1q \leq \frac2n$. By Sobolev estimates, if $\alpha=n(\frac12-\frac1q)$
\begin{multline}\label{2.73}
\|\big(1-\beta(P/\la)\big)(-\Delta_g-\la^2+i \e(\la)\la)^{-1}f\|_q \\
\le \|(-\Delta_g)^{\frac\alpha2}\big(1-\beta(P/\la)\big)(-\Delta_g-\la^2+i \e(\la)\la)^{-1}f\|_2.
\end{multline}
Since the symbol of the operator on the right side of \eqref{2.73} satisfies 
\begin{equation}\label{2.74}
\tau^\alpha \big(1-\beta(\tau/\la)\big)(\tau^2-\la^2+i \e\la)^{-1}\le \la^{\alpha-2},
\end{equation}
by the spectral theorem, this yields the bounds
\begin{equation}\label{2.75}
\|I\|_q\le \la^{n(\frac12-\frac1q)-2}\|f\|_2,
\end{equation}
which are better than 
%the right side of 
those in
\eqref{2.62}, since we are assuming that $\e<\la/2$.

If we combine \eqref{2.66}, \eqref{2.69}, \eqref{2.72}, and \eqref{2.75}, we conclude that for $\la\ge\Lambda$ with $\Lambda=\max\{\Lambda_1,\Lambda_2\}$,
we have
\begin{multline}\label{2.76}
\|(H_V-\la^2+ i\e\la)^{-1}f\|_{L^q(\R^n)}
\le C \,\la^{n(1/2-1/q)-3/2} \,\e^{-1/2}\|f\|_{L^2(\R^n)} \\
+\frac12 \, \| (H_V-\la^2+i\e\la)^{-1}f\bigr\|_{L^q(\R^n)}.
\end{multline}
Note that as a consequence of Lemma~\ref{Sobv}, for $q$ satisfying \eqref{2.61}, we have 
\begin{equation}\label{2.77}
\bigl\| (H_V-\la^2+i\e\la)^{-1}f\bigr\|_{L^q(\R^n)}
<\infty \quad \text{if } \, \, f\in L^2(\R^n).
\end{equation}

Thus \eqref{2.76} and \eqref{2.62} are equivalent, and so the proof of Theorem~\ref{big} is complete.
\end{proof}

Now we shall prove \eqref{1.4} for the region $\Re \zeta< \Lambda^2$ by using the results we have just proved.
Note that by spectral theorem as well as \eqref{below}, we have
\begin{multline}\label{spectral}
\bigl\|(H_V-\zeta)^{-1}f\bigr\|_{L^2(\R^n)}
\le C_{\Lambda, \delta}
\bigl\|(H_V-\Lambda^2+i\Lambda^2/4)^{-1}f\bigr\|_{L^2(\R^n)}, \\
 \text{if } \, \, \Re \zeta < \Lambda^2, \,\,\,\text{and}\,\,\,\mathrm{dist }(\zeta, [-N_0,+\infty))
 \ge \delta.
\end{multline}
On the other hand, by a simple interpolation argument along with the trivial $L^2$ estimates, \eqref{qm}  implies
\begin{equation}\label{low}
\bigl\|(H_V-\Lambda^2+i\Lambda^2/4)^{-1}f\bigr\|_{L^q(\R^n)}
\le C_{\Lambda} \|f\|_{L^2(\R^n)}, 
 \text{if } \, \, \tfrac{2n}{n-1}<q<\tfrac{2n}{n-3}.
\end{equation}
By duality, this yields
\begin{multline}\label{low1}
\bigl\|(H_V-\zeta)^{-1}f\bigr\|_{L^2(\R^n)}
\le C
\|f\|_{L^p(\R^n)}, \\
 \text{if } \,\,  \tfrac{2n}{n+3}<p<\tfrac{2n}{n+1}, \, \, \Re \zeta < \Lambda^2, \,\,\,\text{and}\,\,\,\mathrm{dist }(\zeta, [-N_0,+\infty))
 \ge \delta.
\end{multline}

To prove \eqref{1.4} when $\Re \zeta< \Lambda^2$, by repeating the previous arguments, it suffice to show that
\begin{equation}\label{2.17}
\bigl\| (H_V-\zeta)^{-1}f\bigr\|_{L^q(\R^n)}
\le C\|f\|_{L^p(\R^n)},
\end{equation}
with $(p,q)$ satisfying \eqref{1.3} and \eqref{2.5}, and $\zeta$ satisfying the conditions in \eqref{low1}.

To exploit this we use the second resolvent formula
\eqref{2.1} to write
\begin{align}\label{2.21}
(H_V-\zeta)^{-1}f
= &(-\Delta-\zeta)^{-1}f
\\
&- (-\Delta-\zeta)^{-1}
\bigl(V_{>N}\cdot (H_V-\zeta)^{-1}f\bigr)
\notag
\\
&- (-\Delta-\zeta)^{-1}
\bigl(V_{\le N}\cdot (H_V-\zeta)^{-1}f\bigr)
\notag
\\
&=I-II-III. \notag
\end{align}

By the uniform Sobolev estimates \eqref{2.2} for the 
unperturbed operator we have
\begin{equation}\label{2.22}
\|I\|_q\le C\|f\|_p,
\end{equation}
as well as
\begin{equation}\label{2.22'}\tag{2.22$'$}
\|II\|_q\le C
\bigl\|V_{>N}\cdot
(H_V-\zeta)^{-1}f
\bigr\|_p \le 
C\|V_{>N}\|_{L^{n/2}}
\cdot
\bigl\| (H_V-\zeta)^{-1}f\bigr\|_{L^q},
\end{equation}
using H\"older's inequality and the fact that $\frac1p-\frac1q=\frac2n $ in the last step. By \eqref{2.20} we can
fix $N$ large enough so that
$C\|V_{>N}\|_{L^{n/2}}<1/2$, yielding the bounds
\begin{equation}\label{2.23}
\|II\|_q<\frac12 \, \bigl\|
(H_V-\zeta)^{-1}f
\bigr\|_q.
\end{equation}

To bound the third term $III$, note that by simple Sobolev estimates, for $q$ satisfying \eqref{2.5} and $\zeta$ satisfying conditions in \eqref{spectral}, we have
\begin{equation}\label{low2}
\bigl\|(-\Delta-\zeta)^{-1}f\bigr\|_{L^q(\R^n)}
\le C_{\Lambda} \|f\|_{L^2(\R^n)}.
\end{equation}
If we combine \eqref{low2} and \eqref{low1}, we conclude that
\begin{equation}\label{2.25}
\|III\|_q\le C_{\Lambda, \delta} N \, \|
f\|_p.
\end{equation}
Thus, \eqref{2.22}, \eqref{2.23} and \eqref{2.25} imply that 
\begin{multline}\label{low3}
\|(H_V-\zeta)^{-1}f\|_{L^q(\R^n)}
\le C
\|f\|_{L^p(\R^n)}+1/2\|(H_V-\zeta)^{-1}f\|_{L^q(\R^n)}, \\
 \text{if } \, \, \Re \zeta < \Lambda^2, \,\,\,\text{and}\,\,\,\mathrm{dist }(\zeta, [-N_0,+\infty))
 \ge \delta.
\end{multline}
By \eqref{2.7}, this implies \eqref{1.4} when $\Re \zeta<\Lambda^2$.

To conclude this section we shall give the proof of Theorem~\ref{lemm}.
\vspace{-0.1cm}
\begin{proof}[Proof of Theorem~\ref{lemm}] Write $V=V_{>N}+V_{\le N}$, we shall focus on the term $V_{\le N}$, since, as in \eqref{2.22'}, we can always fix $N$ large enough so that  
\begin{equation}\label{largen}
\|(-\Delta-\la^2+i\e\la)^{-1}\, V_{>N} f\|_{q}\le 1/4 \|f\|_q.
\end{equation}
Recall that for $\la\ge 1$ the Euclidean resolvent kernels equal 
\begin{equation}\label{kernel}
\begin{aligned}
(-\Delta-\la^2+i\e\la)^{-1}(x,y)&=(2\pi)^{-n}\int_{\R^n} \frac{e^{i\langle x-y,\xi \rangle}}{|\xi|^2-\la^2+i\e\la} d\xi \\
&=(2\pi)^{-n}\la^{n-2} \int_{\R^n} \frac{e^{i\la\langle x-y,\xi \rangle}}{|\xi|^2-1+i\e/\la} d\xi.
\end{aligned}
\end{equation}

Fix a real-valued Littlewood-Paley bump function $\beta\in C_0^\infty((1/2,2))$ satisfying 
\begin{equation}\label{bump}
1=\sum_{-\infty}^\infty \beta(2^{-j}s)\, \, \text{for } \, \, 
s>0, \quad \text{and } \, \beta(s)=1, \, \, s\in [3/4,5/4].
\end{equation}
Note that
 is straightforward to check that 
 %the multiplier  
 $m(\xi)=(1-\beta(|\xi|))(|\xi|^2-1+i\e/\la)^{-1}$ is a symbol of order $-2$, i.e.,
$(\frac{d}{d r})^j  m(r)\le C_\alpha (r^2+1)^{-2-j}$. Therefore, by a simple integration by parts argument, 
%the kernel $K_0(x,y)$ in this case satisfies
\begin{equation}\label{k0}
\begin{aligned}
K_0(x,y)&=(2\pi)^{-n}\la^{n-2} \int_{\R^n} \bigl(1-\beta(|\xi|)\bigr)\frac{e^{i\la\langle x-y,\xi \rangle}}{|\xi|^2-1+i\e/\la} d\xi \\ &=
O\bigl( |x-y|^{2-n}(1+\la|x-y|)^{-N}\bigr), \,\,\forall \,\,N>0.
\end{aligned}
\end{equation}
Thus, by Young's inequality
\begin{equation}\label{k01}
\|K_0 (V_{\le N} f)\|_{q}\le  C\la^{-2} \|V_{\le N}  f\|_q \le C\la^{-2} N\|  f\|_q.
\end{equation}
By choosing $\Lambda_1$ such that $ C\Lambda_1^{-2} N \le 1/12$,  we have
\begin{equation}\label{k02}
\|K_0 (V_{\le N} f)\|_{q} \le 1/12\|  f\|_q, \,\,\text{if}\,\, \la\ge\Lambda_1.
\end{equation}
Here in \eqref{k01} and \eqref{k02} we are abusing notation a bit by letting $K_0$ denote the integral operator
with kernel $K_0(x,y)$ and we shall use similar notation in what follows.

On the other hand, if $|\xi|\approx 1$, by stationary phase methods or the Fourier transform formula for the sphere, we have
\begin{equation}\label{k1}
\begin{aligned}
K_1(x,y)&=
(2\pi)^{-n}\la^{n-2} \int_{\R^n} \beta(|\xi|)\frac{e^{i\la\langle x-y,\xi \rangle}}{|\xi|^2-1+i\e/\la} d\xi \\&=\begin{cases}
O(\la^{n-2}),\,\,\text{if}\,\, |x-y|< \la^{-1} \\
\sum_\pm e^{\pm i\la|x-y|} c_\pm(|x-y|),\,\,\text{if}\,\, |x-y|\ge \la^{-1}, 
\end{cases}
\end{aligned}
\end{equation}
where $|\frac{d^j}{ds^j}c_\pm(s)|\le B_j \la^{\frac{n-3}{2}}s^{-\frac{n-1}{2}-j}$ .

Now we split the kernel $K_1(x,y)$ as 
$$K_1(x,y)=\sum_{j=0}^\infty K_1^j(x,y),
$$
where 
$$ K_1^j(x,y)=\beta(|x-y|\la 2^{-j})K_1(x,y), \,\,\text{if}\,\, j>0,
$$
and 
$$ K_1^0(x,y)=\beta_0(|x-y|\la)K_1(x,y),
$$
with 
$$ \beta_0(s)=(1-\sum_{j=1}^\infty \beta (2^{-j} s)).
$$
Note that as a consequence of \eqref{k1} and Young's inequality, 
\begin{equation}\label{k11}
\|K_1^j (V_{\le N} f)\|_{q}\le  C\la^{-2}2^{j\frac{n+1}{2}} \|V_{\le N}  f\|_q \le C\la^{-2}2^{j\frac{n+1}{2}} N\|  f\|_q,
\end{equation}
if $2^j\le 2^{j_0} \approx \la^{\frac{2}{n+1}}$. Thus,
by choosing $\Lambda_2$ such that $ 2C\Lambda_2^{-1} N \le 1/12$,  we have
\begin{equation}\label{k12}
\sum_{j=0}^{j_0}\|K_1^j (V_{\le N} f)\|_{q} \le 1/12\|  f\|_q, \,\,\text{if}\,\, \la\ge\Lambda_2.
\end{equation}

On the other hand, for each fixed $j>0$, the operator $K_1^j$ satisfies the following bounds.

\begin{proposition}\label{oscillatory}
Let $n\ge 3$, and suppose that
\begin{equation}\label{stein}
1\le p \le 2, \quad q\ge(n+1)p^\prime/(n-1).
\end{equation}
Then
\begin{equation}\label{kj}
\|K_1^j f\|_{L^q(\R^n)}\le C \la^{-2+n(\frac1p-\frac1q)}2^{j(\frac{n+1}{2}-\frac np)} \|f\|_{L^p(\R^n)}.
\end{equation}
\end{proposition}
When $ q=(n+1)p^\prime/(n-1)$, \eqref{kj} is essentially Lemma 5.4 in \cite{sogge86}, which can be proved by standard change of scale arguments and an application of Stein's oscillatory integral theorem. While if $q=\infty$, the above estimates follows from Young's inequality noticing that $K_1^j$ is supported in the set $|x-y|\approx \la^{-1}2^j$. The remaining inequalities in \eqref{kj} now follow from interpolation. 

The intersection of the two regions \eqref{1.3} and \eqref{stein} are pairs of exponents $(p,q)$ satisfying
$ \tfrac1p-\tfrac1q=\tfrac2n, \,\,\,   \tfrac{2n^2}{n^2+n+2}<p<\tfrac{2n}{n+1}
$.
By applying \eqref{kj} along the uniform Sobolev line $\tfrac1p-\tfrac1q=\tfrac2n$ and duality, we have for $(p,q)$ satisfying \eqref{1.3},
\begin{equation}\label{kj1}
\|K_1^j f\|_{L^q(\R^n)}\le C 2^{-j\delta(p)} \|f\|_{L^p(\R^n)}, 
\end{equation}
where $\delta(p)>0$ is a fixed constant which depends on $p$.

A combination of \eqref{kj1} and H\"older's inequality yields
\begin{equation}\label{kj3}
\|K_1^j (V_{\le N} f)\|_{q}\le  C 2^{-j\delta(p)} \|V_{\le N}  f\|_p \le  C 2^{-j\delta(p)} \|V\|_{n/2}\|  f\|_q.
\end{equation}
For each fixed $p$ with $\frac1p-\frac1q=\frac2n$, if we choose $\Lambda_3$ large enough such that $$ (1-2^{-\delta(p)})^{-1}C\Lambda_3^{-\frac{2\delta(p)}{n+1}} \|V\|_{n/2} \le 1/12,$$ after summing over $2^j \ge 2^{j_0}\approx \la^{\frac{2}{n+1}}$,
\begin{equation}\label{kj4}
\sum_{j=j_0}^{\infty}\|K_1^j (V_{\le N} f)\|_{q} \le 1/12\|  f\|_q, \,\,\text{if}\,\, \la\ge\Lambda_3.
\end{equation}
If we combine \eqref{largen}, \eqref{k02}, \eqref{k12} and \eqref{kj4}, we conclude that for $\la\ge \max(\Lambda_1, \Lambda_2, \Lambda_3)$ we have
\begin{equation}\label{kk}
\|(-\Delta-\la^2+i\e\la)^{-1}\, (V f)\|_{q}\le 1/2 \|f\|_q, \,\,\text{if}\,\, \tfrac{2n}{n-1}<q<\tfrac{2n}{n-3},
\end{equation}
which completes the proof of Theorem~\ref{lemm}.
\end{proof}

\noindent\textbf{Remark.}
The above proof of Theorem~\ref{lemm} does not make full use of the condition that $V\in L^{n/2}(\R^n)$ and
it can be easily adapted to prove bounds for certain potentials in $L^{n/2}_{\text{loc}}$. For example, in \eqref{k01} and \eqref{k11}, instead of using $L^q\rightarrow L^q$ type estimates, we can use their $L^p\rightarrow L^q$ analogs with $\frac1p-\frac1q=\frac2n$, and exploit the locally supported condition for the operators $K_0$ and $K_1^j$ to remove the $L^\infty$ condition on $V$ used there.
%,  and the decay condition as $|x|\rightarrow \infty$ can be weakened. 

Specifically, if we assume that
\begin{equation}\label{newassump}
\lim_{r\rightarrow 0}\sup_x\|V\|_{L^{n/2}(B_x(r))}\rightarrow 0
\end{equation}
as well as 
\begin{equation}\label{newassump1}
\sup_x\|V\|_{L^{n/2}(B_x(R))}\le C_\e R^\e, \, \, \forall \,\,\e>0, \, \, R>1,
\end{equation}
then one could decompose the resolvent operator $(-\Delta-\la^2+i\e\la)^{-1}$ as a sum of dyadically localized operators and repeat the above arguments to get the same conclusion.  Also, the condition
 \eqref{newassump} itself ensures that the operator $H_V$ defines a self adjoint operator which is bounded from below. Thus, it is possible to get the same results in Theorem~\ref{unifSob} assuming that $V$ satisfies \eqref{newassump} and \eqref{newassump1}. 
 %As a example, 
 Recall that in Ionescu and Jerison~\cite{2003absence}, a special case of the conditions the authors used to guarantee the absence of positive eigenvalues was that 
$V\in L_{loc}^{n/2}(\R^n)$ and 
\begin{equation}
\nonumber \lim_{R\rightarrow\infty} \|V\|_{L^{n/2}(R<|x|<2R)}=0,
\end{equation}
which implies \eqref{newassump} and \eqref{newassump1} with $R^\e$ replaced by $\log R$.

On the other hand, by an interpolation theorem of Stein and Weiss, i.e., Chapter V, Theorem 3.15 in \cite{stein2016introduction}, it is straightforward to show that, e.g., \eqref{kj3} still holds with $\|V\|_{L^{n/2}}$ 
replaced by $\|V\|_{L^{n/2, \infty}}$.   Thus the conditions in \eqref{newassump} and \eqref{newassump1} can further be weakened to 
\begin{equation}\label{newassump2}
\lim_{r\rightarrow 0}\sup_x\|V\|_{L^{n/2,\infty}(B_x(r))}\rightarrow 0
\end{equation}
as well as 
\begin{equation}\label{newassump3}
\sup_x\|V\|_{L^{n/2,\infty}(B_x(R))}\le C_\e R^\e, \, \, \forall \,\,\e>0, \, \, R>1.
\end{equation}

\newsection{Strichartz estimates for Schr\"odinger operators with singular potentials in $\mathbb{R}^n$.}\label{stsec}

In this section we shall prove Theorem~\ref{hvthm}. We shall first give the proof of \eqref{localstrichartz}, which is an easy consequence of Duhamel's principle and a bootstrap argument.

If $V\in L^{n/2}(\mathbb{R}^n)+L^{\infty}(\mathbb{R}^n)$, by adapting the arguments in the appendix of \cite{blair2020uniform}, one can show that $H_V=-\Delta+V$ defines a self-adjoint operator which 
is bounded from below. If necessary, we may add a constant to $V$ so that $H_V\ge 0$.
 This will not affect
our estimates, since, if we, say, add the constant
$N$ to $V$ the two different Schr\"odinger operators
will agree up to a factor $e^{\pm itN}$.
By spectral theorem, we have
\begin{equation*}
\|e^{-itH_V}\|_{L^2(\R^n)\to L^2(\R^n)}=O(1),\,\, \forall\,\, t\in \R.
\end{equation*}
Thus, it suffices to prove \eqref{localstrichartz} for the other endpoint, i.e.,
\begin{equation}\label{local1}
\bigl\|e^{-itH_V}f\bigr\|_{L^2_tL^{\frac{2n}{n-2}}_x([0,1]\times \R^n)}
\lesssim \|f\|_{L^2(\mathbb{R}^n)}.
\end{equation}

To proceed, note that for $V\in L^{n/2}(\mathbb{R}^n)+L^{\infty}(\mathbb{R}^n)$, as in \eqref{2.18}--\eqref{2.20}, we can always write $V=V_{\le N}+V_{>N}$,
with $$\|V_{\le N}\|_{L^\infty}\le N,$$ while $$\|V_{> N}\|_{L^{n/2}}=\delta(N) \rightarrow 0$$ as $N\rightarrow \infty$.

By Duhamel's formula, we can write
\begin{equation}\label{duhamel}
\begin{aligned}
e^{-itH_V}f=&e^{it\Delta}f+i\int_0^t e^{i(t-s)\Delta}V_{\le N}e^{-isH_V}fds+i\int_0^t e^{i(t-s)\Delta}V_{>N}e^{-isH_V}fds \\
=& I+II+III.
\end{aligned}
\end{equation}

Note that, by the Keel-Tao \cite{KT} theorem, we have
\begin{equation}\label{keel1}
\bigl\|e^{it\Delta}f\bigr\|_{L^2_tL^{\frac{2n}{n-2}}_x([0,1]\times \R^n)}
\le C\|f\|_{L^2(\mathbb{R}^n)},
\end{equation}
as well as 
\begin{equation}\label{keel2}
\bigl\|\int_0^t e^{i(t-s)\Delta}F(s,\cdot)ds\bigr\|_{L^2_tL^{\frac{2n}{n-2}}_x([0,1]\times \R^n)}
\le C\|F\|_{L_t^2L_x^{\frac{2n}{n+2}}([0,1]\times \R^n)}.
\end{equation}

Thus, the first term $I$ is bounded by the right side of \eqref{local1}, and if we choose $N$ large enough such that $C\delta(N) \le 1/2$ for the constant $C$ appeared in \eqref{keel2}, we have
\begin{equation}\label{III}
\begin{aligned}
\|III\|_{L^2_tL^{\frac{2n}{n-2}}_x([0,1]\times \R^n)}\le& C\|V_{>N}e^{-itH_V}f\|_{L_t^2L_x^{\frac{2n}{n+2}}(\mathbb{R}^n)} \\
\le& C \|V_{>N}\|_{L^{n/2}(\R^n)} \bigl\|e^{-itH_V}f\bigr\|_{L^2_tL^{\frac{2n}{n-2}}_x([0,1]\times \R^n)} \\
\le& 1/2 \bigl\|e^{-itH_V}f\bigr\|_{L^2_tL^{\frac{2n}{n-2}}_x([0,1]\times \R^n)}.
\end{aligned}
\end{equation}

To estimate $II$ we shall use Minkowski inequality and \eqref{keel1}. More specifically,
\begin{equation}\label{II}
\begin{aligned}
\|II\|_{L^2_tL^{\frac{2n}{n-2}}_x([0,1]\times \R^n)}\le& C\int_0^1\| e^{it\Delta}(e^{is\Delta}V_{\le N}e^{-isH_V})\|_{L^2_tL^{\frac{2n}{n-2}}_x([0,1]\times \R^n)}ds \\
\le&  C\int_0^1\|e^{is\Delta}V_{\le N}e^{-isH_V}f\|_{L^2_x(\R^n)}ds \\
\le&  CN\int_0^1\|e^{-isH_V}f\|_{L^2_x(\R^n)}ds \\
\le& CN\|f\|_{L^2(\R^n)}.
\end{aligned}
\end{equation}

If $f\in H^1(\R^n)$, then by Sobolev estimates and (6.9) in the appendix of \cite{blair2020uniform}, $e^{itH_V}f\in L^{\frac{2n}{n-2}}$, which implies
$$\bigl\|e^{-itH_V}f\bigr\|_{L^2_tL^{\frac{2n}{n-2}}_x([0,1]\times \R^n)}<\infty.
$$
Thus, a combination of \eqref{keel1}, \eqref{III}, \eqref{II} and a bootstrap argument gives us
\begin{equation}\label{dense}
\bigl\|e^{-itH_V}f\bigr\|_{L^2_tL^{\frac{2n}{n-2}}_x([0,1]\times \R^n)}\le C_V \|f\|_{L^2(\R^n)}, \,\,\text{if}\,\, f\in H^1(\R^n).
\end{equation}
Since $H^1$ is dense in $L^2$, the proof of \eqref{local1} is complete.

The proof of \eqref{strichartz} requires more work since we can not bound the term $II$ as before if the support of $s$ is unbouned. To proceed, we shall follow the strategy in a recent work \cite{huang2020quasimode} by the authors and prove an analogous dyadic estimates which will allow us to obtain \eqref{strichartz}. Also, we have to show that the Littlewood-Paley estimates for $H_V$ are valid for the exponents $q$ as in \eqref{i.2}.  This was done in the appendix of \cite{huang2020quasimode} in the setting of compact manifolds
and the same argument can be used to handle the Euclidean space $\R^n$.
% however, this is given in the appendix of \cite{huang2020quasimode}. Although the arguments there are for $H_V$ on compact manifolds, they work equally well for the Euclidean space $\R^n$ without essential change.

As in \cite{huang2020quasimode}, the proof of dyadic variants of \eqref{i.2} rely on certain microlocalized ``quasimode" estimates for the unperturbed scaled Schr\"oding operators with a damping term,
\begin{equation}\label{damp}
i\la\partial_t+\Delta+i\e\la.
\end{equation}
Unlike the case of compact manifold, we need to allow the damping term $\e\la$ to be arbitrary small and prove results that are uniform in $\e$ in order to get global Strichartz estimates. Furthermore, since the Littlewood-Paley operators associated with $-\Delta$ may not be compatible with the corresponding ones for $H_V=-\Delta+V(x)$ with V singular, we shall also introduce the Littlewood-Paley operators acting on the time variable 
$$\beta(-D_t/\la)h(x)=(2\pi)^{-1}
\int_{-\infty}^\infty e^{it\tau}
\beta(-\tau/\la) \, \Hat h(\tau)\, d\tau,
$$
with $\beta$ defined as in \eqref{bump}.

More specifically, to prove \eqref{strichartz} our main estimates will concern solutions of the scaled inhomogeneous Schr\"odinger equation with damping term
\begin{equation}\label{i.11}
(i\la \partial_t+\Delta_g+i\e\la)w(t,x)=F(t,x), \quad
w(0,\cd)=0.
\end{equation}
It will be convenient to assume that the ``forcing term'' here satisfies
\begin{equation}\label{i.12}
F(t,x)=0, \quad t\notin [0,\e^{-1}].
\end{equation}

The result that we shall need in order to prove \eqref{strichartz} is the following.
\begin{theorem}\label{lemmm}  Let $n\ge 3$, suppose
$F$ satisfies the support assumption in \eqref{i.12}
and $w$ solves \eqref{i.11}.  Then for $\la\ge 1$, we have
\begin{equation}\label{i.13}
\bigl\|\beta(-D_t/\la)w\bigr\|_{L^2_tL^{\frac{2n}{n-2}}_x(\R\times \R^n)}
\lesssim \la^{-1/2} \e^{-1/2}\|F\|_{L^2_{t,x}([0,\e^{-1}]\times \R^n)},
\end{equation}
and also
\begin{equation}\label{i.14}
\bigl\|\beta(-D_t/\la)w\bigr\|_{L^2_tL^{\frac{2n}{n-2}}_x(\R\times \R^n)}
\lesssim 
\|F\|_{L^{2}_tL^{\frac{2n}{n+2}}_x([0,\e^{-1}]\times \R^n)}.
\end{equation}
Additionally, if for all time $t$, $\supp F(t, \cdot)\subset B_R$, with $B_R$ being a fixed ball of radius $R\ge1$ in $\R^n$ centered at origin,
we have 
\begin{equation}\label{i.15}
\bigl\|\beta(-D_t/\la)w\bigr\|_{L^2_tL^{\frac{2n}{n-2}}_x(\R\times B_R)}
\le C_R \la^{-1/2}
\|F\|_{L^{2}_tL^{2}_x([0,\e^{-1}]\times \R^n)}.
\end{equation}
\end{theorem}

Inequalities \eqref{i.13} and \eqref{i.14} are analogs of those in Theorem~1.2 in \cite{huang2020quasimode}, which, as we shall see later, can be proved by using a similar argument.  Inequality \eqref{i.15} is the main new ingredient, which will allow us to deal with forcing terms involving bounded and compactly supported potentials.  Also, compared with Theorem~1.2 in \cite{huang2020quasimode}, we only consider the special case $p=2$, $q=\frac{2n}{n-2}$ here for the sake of simplicity since we are assuming that $n\ge3 $ where the endpoints are always admissible . %while the fact that $p<q$ will also help us to justify a localization argument later in the proof of \eqref{i.15}.
In other words, this endpoint Strichartz estimate implies all the others by interpolating with the trivial $L^2$-estimate.

Before giving the proof of Theorem \eqref{lemmm}, which we shall postpone until the end of the section, let us see how we can use the above inequalities to prove \eqref{strichartz}. 
Recall that under the assumption \eqref{def}, $H_V$ defines a self-adjoint operator and, as noted before, we may assume that
\begin{equation}\label{d.14}
H_V\ge 0.
\end{equation}
%as before since it will not affect our estimates.

To proceed, we shall require the following multiplier bounds associated to the operator $H_V$.
\begin{proposition}\label{multiplier}
Let $V\in L^{n/2}(\R^n)$ be real valued, and suppose that 
\begin{equation}\label{multi}
1<q<\infty \,\,for\,\, n=3,4\,\, or \,\,2n/(n+4)<q<2n/(n-4) \, \, n\ge5.
\end{equation}
Then if $m\in C^\infty (\R_+)$ is a Mikhlin-type multiplier, i.e.,
\begin{equation}\label{h.1}
|\partial_\tau^j m(\tau)|\le C(1+\tau)^{-j}, \tau >0, \, \, \,
0\le j\le n/2 +1,
\end{equation}
we have
\begin{equation}\label{multineq}
\bigl\|m(\sqrt{H_V})f\bigr\|_{L^q(\R^n)}
\le C_q\|f\|_{L^q(\R^n)} ,
\end{equation}
as well as
\begin{equation}\label{lpest}
 \|h\|_{L^q(\R^n)} \le C_{q,V}\, 
\|\beta_0(\sqrt{H_V})h\|_{L^q(\R^n)}
+
\bigl\| \, \bigl(\, \sum_{k=1}^\infty
\, \bigl|\beta(\sqrt{H_V}/2^k)h\bigr|^2\, \bigr)^{1/2} \,
\bigr\|_{L^q(\R^n)},
\end{equation}
for $q$, $V$ as above and $\beta$ being the Littlewood-Paley bump function satisfying \eqref{bump}.
\end{proposition}

The proof of \eqref{multineq} was given in the appendix of \cite{huang2020quasimode} in the compact manifold case.  We can
repeat the arguments and see that, to obtain \eqref{multineq} in the Euclidean setting,
%adapt the same argument for the Euclidean space, 
it suffices to check the following two properties hold for the operator $H_V$ in $\R^n$.

 i). Finite propagation speed for the wave equation associated to $H_V$, that is, if $u,v\in L^2(\R^n)$ and $\dist(\text{supp }u,\text{supp }v)=R$ then
\begin{equation}\label{h.2}
\bigl(u, \, \cos t\sqrt{H_V} \, v\bigr)=0, \, \, |t|<R.
\end{equation}
By a result of Coulhon and Sikora \cite{CouS},
\eqref{h.2} is valid when $H_V$ is nonnegative,
self-adjoint and $V\in L^1_{loc}(\R^n)$.  Alternately, one can use arguments from \cite{BSS} to show this for the potentials that
we are considering.

ii). For $q$ satisfying \eqref{1.6} and $\beta$ defined as in \eqref{bump}, we have the Bernstein type (dyadic Sobolev)
estimates
\begin{multline}\label{h.3}
\|\beta(\sqrt{H_V}/\la) u\|_{L^{q}(\R^n)}
\le C\la^{n(\frac12-\frac1{q})}\|u\|_{L^2(\R^n)}, \, \,
\la \ge 1,
\\
\text{and } \, \,
\|\beta_0(\sqrt{H_V})u\|_{L^{q}(\R^n)}\le C\|u\|_{L^2(\R^n)},
\, \, \text{if } \, \beta_0(s)=1-\sum_{k=1}^\infty
\beta(2^{-k}s).
\end{multline}
%However, \eqref{h.3} 
This, in turn, 
is a direct consequence of \eqref{qm} and the spectral theorem if we let $u=\beta(\sqrt{H_V}/\la) u$ and $\e=\la$ there.

The fact that the conditions $i)$ and $ii)$ implies \eqref{multineq} now follows from the same procedure as in the appendix of \cite{huang2020quasimode}.   Also given \eqref{multineq}, the Littlewood-Paley estimate \eqref{lpest} can be proved using a standard argument involving Radamacher functions (see e.g., \cite[p. 21]{SFIO2}).

In view of the Littlewood-Paley inequalities \eqref{lpest}, by using a standard argument involving Minkowski's inequality (see, e.g., (3.33) in \cite{huang2020quasimode}), the Strichartz estimate in \eqref{strichartz} is equivalent to the following:
\begin{multline}\label{d.15}
\bigl\|e^{-itH_V}\rho(\sqrt{H_V}/\la)f\bigr\|_{L^p_tL^q_x(\R\times \R^n)}
\lesssim \|\rho(\sqrt{H_V}/\la)f\|_{L^2(\R^n)}, \\ \,\,\text{if}\,\, \rho\in C_0^\infty((9/10, 11/10))\,\,\text{is fixed},\,\,\text{and}\,\,\la>\Lambda,
\end{multline}
assuming that $\Lambda=\Lambda(V)$ sufficiently large.

We choose
this interval $(9/10, 11/10)$ as the support of $\rho$ since we are assuming that the 
Littlewood-Paley bump function arising in
Theorem~\ref{lemmm} satisfies
\begin{equation}\label{d.2}
\beta(s)=1 \quad \text{on } \, \,
[3/4,5/4] \quad
\text{and } \, \, \mathrm{supp } \,\beta
\subset (1/2,2).
\end{equation}

Since, by the spectral theorem
\begin{equation}\label{spec}
\|e^{-itH_V}\|_{L^2(\R^n)\to L^2(\R^n)}=1,
\end{equation}
the estimate trivially holds for $p=\infty$ and
$q=2$.  Therefore, by interpolation, 
since we are assuming that $n\ge3$ it 
suffices to prove the estimate for the other endpoint,
i.e.,
\begin{equation}\label{3.17}
\bigl\|e^{-itH_V}\rho(\sqrt{H_V}/\la)f\bigr\|_{L^2_tL^{2n/(n-2)}_x(\R\times \R^n)}
\le C\|\rho(\sqrt{H_V}/\la)f\|_{L^2(\R^n)}, \,\,\text{if}\,\,\la>\Lambda.
\end{equation}

To be able to use \eqref{i.13} and \eqref{i.14}, we note that, after rescaling, \eqref{3.17} is equivalent to the statement that
\begin{multline}\label{3.18}
\bigl\|e^{-it\la^{-1}H_V}\rho(\sqrt{H_V}/\la)f\bigr\|_{L^2_tL^{2n/(n-2)}_x([-\e^{-1},\e^{-1}]\times \R^n)} \\
\le C\la^{1/2}\|\rho(\sqrt{H_V}/\la)f\|_{L^2(\R^n)},\,\,\forall\,\,0<\e<1, \,\,\text{if}\,\,\la>\Lambda.
\end{multline}
where $C=C(n, V)$ is a uniform constant independent of $\e$. 

By \eqref{spec}, this is equivalent to showing that 
whenever $$\eta\in C^\infty((0,1)),\,\,0<\e<1$$ are fixed we have
\begin{multline} \label{3.19}
\bigl\|w\bigr\|_{L^2_tL^{2n/(n-2)}_x(\R\times M)}
\le C\la^{1/2}\|\rho(\sqrt{H_V}/\la)f\|_{L^2(M)}, \,\,\text{if}\,\,\la>\Lambda, \\ \text{with}\,\,\,
w(t,x)= \eta(\e t)\cdot 
e^{-it\la^{-1}H_V}\rho(\sqrt{H_V}/\la)f.
\end{multline}

To proceed we need the following simple lemma.

\begin{lemma}\label{Verror}  Let $n\ge3$ and let $w$ be as in
\eqref{3.19} with $\eta\in C^\infty_0((0,1))$ and $\rho\in C_0^\infty$ as in \eqref{d.15}.
Then for large enough $\la$ and each $N=1,2,\dots$
we have the uniform bounds
\begin{equation}\label{3.20}
\bigl\| \, (I-\beta(-D_t/\la))w \,
\bigr\|_{L^2_tL^{\frac{2n}{n-2}}_x(\R \times \R^n)}
\le C_N \la^{-N}\|\rho(\sqrt{H_V}/\la)f\|_{L^2(\R^n)}.
\end{equation}
\end{lemma}

This is essentially the Euclidean version of Lemma 3.1 in \cite{huang2020quasimode}, and so we can repeat the arguments there to prove \eqref{3.20} with minor modifications. For completeness, we include the details below.
\begin{proof}  First note the Fourier transform
of $t\to \eta(\e t)e^{-it\la^{-1}\mu^2}$ is
$\e^{-1}\Hat \eta(\frac{\tau+\mu^2/\la}{\e})$.  Consequently,
\begin{equation}\label{d.4}
\bigl(I-\beta(-D_t/\la)\bigr)w(t,x)
=  a(t;\sqrt{H_V}) \rho(\sqrt{H_V}/\la)f(x),
\end{equation}
where
\begin{equation}\label{d.5}
a(t;\mu)=(2\pi)^{-1}
\int_{-\infty}^\infty e^{it\tau}
\e^{-1}\Hat \eta(\frac{\tau+\mu^2/\la}{\e}) \, 
\big(1-\beta(-\tau/\la)\bigr) \, d\tau.
\end{equation}

Since for $q_0=\frac{2n}{n-2}$, by (6.7) in \cite{blair2020uniform}, 
the following Sobolev estimates are valid
\begin{equation}\label{d.s}
\|u\|_{L^{q_0}(\R^n)}\lesssim \| \, (I+H_V)^{1/2}u\, \|_{L^2(\R^n)}.
\end{equation}
Therefore, by the spectral theorem,
\begin{equation}\label{d.6}
\bigl\| a(t;\sqrt{H_V}) \rho(\sqrt{H_V}/\la)f\bigl\|_{L^{q_0}(\R^n)}\lesssim 
\la \bigl\| a(t;\sqrt{H_V}) \rho(\sqrt{H_V}/\la)f \bigl\|_{L^2(\R^n)}.
\end{equation}

Next, since $2\le q_0<\infty$, by Minkowski's inequality and Sobolev's theorem for $\R$ we therefore have
\begin{align*}
\|(I-\beta(-D_t/\la))w\|_{L^2_tL^{q_0}_x(\R\times M)} &\lesssim 
\la \bigl\|  a(t;\sqrt{H_V}) \rho(\sqrt{H_V}/\la)f \bigl\|_{L^{q_0}_tL^2_x(\R\times \R^n)}
\\
&\le \la \bigl\| a(t;\sqrt{H_V}) \rho(\sqrt{H_V}/\la)f \bigl\|_{L^2_xL^{q_0}_t(\R\times \R^n)}
\\
&\lesssim \la \bigl\|  \, |D_t|^{1/2-1/q_0}a(t;\sqrt{H_V}) \rho(\sqrt{H_V}/\la)f \bigl\|_{L^2_xL^2_t(\R\times \R^n)}.
\end{align*}
By the spectral theorem, we conclude that
\begin{multline}\label{d.7}
\|(I-\beta(-D_t/\la)w\|_{L^2_tL^{q_0}_x(\R\times \R^n)}
\\
\lesssim \la \, \bigl(\sup_{\mu\in[9\la/10, 11\la/10]}\bigl\| \, |D_t|^{1/2-1/p}a(t;\mu)\bigr\|_{L^2_t(\R)}\bigr) \cdot \|\rho(\sqrt{H_V}/\la)f\|_{L^2(\R^n)}.
\end{multline}

Next,  by Plancherel's theorem, \eqref{d.2} and \eqref{d.5},
\begin{align*}
\| \, |D_t|^{1/2-1/p}a(t;\mu)\|_{L^2_t(\R)}^2 &=(2\pi)^{-1} \int_{-\infty}^\infty
|\tau|^{1-2/p} \, \bigl|\e^{-1}\Hat \eta(\frac{\tau+\mu^2/\la}{\e})\bigr|^2 \, \bigl|(1-\beta(-\tau/\la))\bigr|^2 \, d\tau
\\
&\lesssim \int_{\tau\notin [-5/\la/4,\, -3\la/4]} |\tau|^{1-2/p} \, |\e^{-1}\Hat \eta(\frac{\tau+\mu^2/\la}{\e})|^2 \, d\tau.
\end{align*}
Note that $|\tau+\mu^2/\la|\approx (|\tau|+\la)$ if if $\tau\notin [-5\la/4,-3\la/4]$ and $\mu\in[9\la/10, 11\la/10]$
, and since $\Hat \eta\in {\mathcal S}(\R)$ the preceding inequality leads to the trivial bounds
\begin{equation}\label{d.8}
\sup \| \, |D_t|^{1/2-1/p}a(t;\mu)\|_{L^2_t(\R)} \lesssim \la^{-N}.
\end{equation}
Combining this inequality with \eqref{d.7} yields \eqref{3.20}.
\end{proof}

We now are able to prove \eqref{3.19}.  For fixed $V\in L^{n/2}(\R^n)$, let us split
$$V=V_{1}+V_{2},
$$
where
\begin{equation}\label{V}
V_{1}(x)=\begin{cases} V(x),\, \, \, \text{if } \, |V(x)|\le \ell 
\, \, \text{and } \, |x|<R \, \, \\ 0, \,\,\text{otherwise}.
\end{cases}
\end{equation}
The assumption $V\in L^{n/2}(\R^n)$ then yields
\begin{equation}\label{d.19}
\|V_{2}\|_{L^{n/2}(\R^n)}=\delta(\ell, R), \quad
\text{with } \, \, \delta(\ell, R)\to 0 \, \, 
\text{as } \, \ell, R \to \infty,
\end{equation}
and we also trivially have
\begin{equation}\label{d.20}
\|V_{1}\|_{L^\infty(\R^n)}\le \ell.
\end{equation}

To use this we note that since $-H_V=\Delta_g-V$
\begin{align*}(i\la\partial_t+\Delta_g+i\la)w
&= (i\la\partial_t-H_V+i\e\la)w +Vw
\\
&= (i\la\partial_t-H_V+i\e\la)w +V_{1}\, w +V_{2} \, w,
\end{align*}
and also $w(0,\cd)=0$.
So we can split
\begin{equation}\label{d.21}
w=\widetilde w+ w_{1} + w_{2},
\end{equation}
where
\begin{equation}\label{d.22}
(i\la\partial_t+\Delta_g+i\e\la)\widetilde w
= (i\partial_t-H_V+i\e\la)w =\widetilde F, \quad 
\tilde w(0,\cd)=0,
\end{equation}
\begin{equation}\label{d.23}
(i\la\partial_t+\Delta_g+i\e\la)w_{1}= V_{1} \, w
=F_{1}, \quad w_{1}(0,\cd)=0,
\end{equation}
and
\begin{equation}\label{d.24}
(i\la\partial_t+\Delta_g+i\e\la)w_{2}= V_{2}\, w
=F_{2}, \quad w_{2}(0,\cd)=0.
\end{equation}
Note that since $w(t,x)=0$, $t\notin (0,\e^{-1})$ each
of the forcing terms $\widetilde F$, $F_{1}$ 
and $F_{2}$ also vanishes for such $t$ which allows
us to apply the estimates in Theorem~\ref{lemmm}
for $\widetilde w$, $w_{1}$ and $w_{2}$.

By \eqref{3.20} and \eqref{d.21} we have for each $N=1,2,\dots$
\begin{align}\label{d.25}
\bigl\|w&\bigr\|_{L^2_tL^{2n/(n-2)}_x(\R\times \R^n)}
\\
&\lesssim \bigl\|\beta(-D_t/\la)w\bigr\|_{L^2_tL^{2n/(n-2)}_x(\R\times \R^n)}+C_N \la^{-N}\|\rho(\sqrt{H_V}/\la)f\|_{L^2(\R^n)} \notag
\\
&\le
\bigl\|\beta(-D_t/\la)\widetilde w\bigr\|_{L^2_tL^{2n/(n-2)}_x(\R\times \R^n)}
+ \bigl\|\beta(-D_t/\la)w_{1}\bigr\|_{L^2_tL^{2n/(n-2)}_x(\R\times \R^n)} \notag
\\
&+
\bigl\|\beta(-D_t/\la)w_{2}\bigr\|_{L^2_tL^{2n/(n-2)}_x(\R\times \R^n)}+ C_N \la^{-N}\|\rho(\sqrt{H_V}/\la)f\|_{L^2(\R^n)}.
\notag
\end{align}
Based on this  we
would obtain \eqref{3.19} if we have the following three inequalities
\begin{equation}\label{d.26}
\bigl\| \beta(-D_t/\la) \widetilde w\bigr\|_{L^2_tL^{2n/(n-2)}_x(\R\times  \R^n)}
\le C\la^{1/2}\|\rho(\sqrt{H_V}/\la)f\|_{L^2( \R^n)},
\end{equation}
as well as
\begin{multline}\label{d.27}
\bigl\| \beta(-D_t/\la) w_{1}\bigr\|_{L^2_tL^{2n/(n-2)}_x(\R\times  \R^n)}
\le \tfrac14\bigl\|w\bigr\|_{L^2_tL^{2n/(n-2)}_x(\R\times \R^n)}\\+C\la^{1/2}\|\rho(\sqrt{H_V}/\la)f\|_{L^2( \R^n)},
\end{multline}
and finally
\begin{equation}\label{d.28}
\bigl\| \beta(-D_t/\la) w_{2}\bigr\|_{L^2_tL^{2n/(n-2)}_x(\R\times  \R^n)}
\le \tfrac14 \bigl\|w\bigr\|_{L^2_tL^{2n/(n-2)}_x(\R\times  \R^n)}.
\end{equation}
Indeed we just combine \eqref{d.25}--\eqref{d.28}
and use a simple bootstrapping argument which is justified since the right side of \eqref{d.28} is finite
by the aforementioned Sobolev estimates for $H_V$.

To prove these three estimates we shall use Theorem~\ref{lemmm},  as we may,  since,
as  mentioned before, the forcing terms in \eqref{d.22},
\eqref{d.23} and \eqref{d.24} obey the support assumption
in \eqref{i.12}.

To prove \eqref{d.26} we note that if $\widetilde F$
is as an \eqref{d.22} then, since $w$ is as in \eqref{3.19}, we have
\begin{multline}\widetilde F(t,x)
= (i\la\partial_t-H_V+i\e\la) \bigl(\eta(\e t)e^{-it\la^{-1}H_V}\rho(\sqrt{H_V}/\la)f(x)\bigr)
\\=i\e \la(\eta'(\e t)+\eta(\e t)) e^{-it\la^{-1}H_V}\rho(\sqrt{H_V}/\la)f(x).\end{multline}
Consequently, we may use the $L^2$-estimate, \eqref{i.13}, in
Theorem~\ref{lemmm} to deduce
that
\begin{align*}
\bigl\| \beta(-D_t/\la) &\widetilde w\bigr\|_{L^2_tL^{2n/(n-2)}_x(\R\times \R^n)}
\\
&\le \la^{-1/2}\e^{-1/2}\| i\e\la(\eta'(\e t)+\eta(\e t))\cdot e^{-it\la^{-1}H_V}\rho(\sqrt{H_V}/\la)f\|_{L^2_{t,x}(\R\times \R^n)}
\\
&\lesssim \la^{1/2} \|\rho(\sqrt{H_V}/\la)f\|_{L^2(\R^n)},
\end{align*}
as desired.

To prove \eqref{d.27} and \eqref{d.28} we shall need to use \eqref{i.14} and \eqref{i.15} in Theorem~\ref{lemmm}.
Note that
$$\frac1{q'}-\frac1q=\frac2n, \quad \text{if } \, \,
q=2n/(n-2), \,  \, q'=2n/(n+2).$$
Consequently if we use \eqref{i.14}, \eqref{d.24}, H\"older's inequality and \eqref{d.19} then
we conclude that we can
fix $\ell_1, R_1$ large enough so that we have
\begin{equation}\label{3.39}
\begin{aligned}\bigl\| \beta(-D_t/\la)  w_{2}\bigr\|_{L^2_tL^{2n/(n-2)}_x(\R\times \R^n)}
&\le C\|V_{2} \, w\|_{L^2_tL^{2n/(n+2)}_x(\R\times  \R^n)}
\\
&\le C\|V_{2}\|_{L^{n/2}( \R^n)} \cdot 
\| w  \|_{L^2_tL^{2n/(n-2)}_x(\R\times  \R^n)}
\\
&\le \tfrac 14 \| w  \|_{L^2_tL^{2n/(n-2)}_x(\R\times  \R^n)},
\end{aligned}
\end{equation}
assuming, as we may, in the last step that, if $\delta(\ell_1, R_1)$ is as
in \eqref{d.19},  $C\delta(\ell_1, R_1)\le \tfrac14$.

Similarly, by repeating the above argument, and using the support condition on $V_1$, we have
\begin{equation}\label{3.40}
\begin{aligned}\bigl\| \beta(-D_t/\la)  w_{1}\bigr\|_{L^2_tL^{2n/(n-2)}_x(\R\times \R^n)}
&\le C\|V_{1} \, w_1\|_{L^2_tL^{2n/(n+2)}_x(\R\times  \R^n)}
\\
&\le C\|V_{1}\|_{L^{n/2}( \R^n)} \cdot 
\| \chi_{R_1}w  \|_{L^2_tL^{2n/(n-2)}_x(\R\times  \R^n)}
\\
&\le C\|V\|_{L^{n/2}( \R^n)}  \| \chi_{R_1}w  \|_{L^2_tL^{2n/(n-2)}_x(\R\times  \R^n)},
\end{aligned}
\end{equation}
where $\chi_{R_1}\in C_0^\infty(\R^n)$ with $\chi_{R_1}\equiv 1$ if $|x|\le R_1$ and vanishes if $|x|>2R_1$.

Although at the moment we do not obtain the right side of \eqref{d.27} as we want, we gain a support function $\chi_R$, which will allow us to use \eqref{i.15}. To see this, 
as in \eqref{d.21}, we shall split
\begin{equation}\label{3.41}
\chi_{R_1}w=\chi_{R_1}\widetilde w+ \chi_{R_1}w_{1} + \chi_{R_1}w_{2},
\end{equation}
with $\widetilde w, w_{1}, w_{2}$ satisfying \eqref{d.22}, \eqref{d.23}, \eqref{d.24} correspondingly, but with a different choice of $\ell, R$ in \eqref{V} which we shall specify later.

By \eqref{d.25} and \eqref{3.40}, we would obtain \eqref{d.27} if we have the following inequalities
\begin{equation}\label{3.42}
\bigl\| \chi_{R_1}\beta(-D_t/\la) \widetilde w\bigr\|_{L^2_tL^{2n/(n-2)}_x(\R\times  \R^n)}
\le C\la^{1/2}\|\rho(\sqrt{H_V}/\la)f\|_{L^2( \R^n)},
\end{equation}
\begin{equation}\label{3.43}
\bigl\|\chi_{R_1} \beta(-D_t/\la) w_{1}\bigr\|_{L^2_tL^{2n/(n-2)}_x(\R\times  \R^n)}
\le \tfrac{1}{8C_1}\bigl\|w\bigr\|_{L^2_tL^{2n/(n-2)}_x(\R\times \R^n)},
\end{equation}
and finally
\begin{equation}\label{3.44}
\bigl\| \chi_{R_1}\beta(-D_t/\la) w_{2}\bigr\|_{L^2_tL^{2n/(n-2)}_x(\R\times  \R^n)}
\le \tfrac{1}{8C_1} \bigl\|w\bigr\|_{L^2_tL^{2n/(n-2)}_x(\R\times  \R^n)},
\end{equation}
with $C_1=C\|V\|_{L^{n/2}( \R^n)}$ as in \eqref{3.40}.

Note that \eqref{3.42} follows directly from \eqref{d.26}, and \eqref{3.44} follows from \eqref{3.39} if we choose $\ell_2, R_2$ large enough such that $R_2\ge R_1$ and 
$C\delta(\ell_2, R_2) <\frac{1}{8C_1}$.

Up until now we have not used the condition that $\la>\Lambda$. We need the condition to obtain \eqref{3.43} with the required small constant on the right side. To see this,
note that by applying \eqref{i.15} for $R=R_2$ and using the fact that $V_1\le \ell_2$ and $\supp V_1\subset B_{R_2}$, we have
\begin{equation}\label{3.45}
\begin{aligned}
\bigl\|\chi_{R_1} \beta(-D_t/\la) w_{1}\bigr\|_{L^2_tL^{2n/(n-2)}_x(\R\times  \R^n)}
&\le C_{R_2}\la^{-1/2} \bigl\|V_1w\bigr\|_{L^2_tL^2_x(\R\times \R^n)}\\
&\le C_{R_2}R_2\ell_2\la^{-1/2} \bigl\|w\bigr\|_{L^2_tL^{2n/(n-2)}_x(\R\times \R^n)}.
\end{aligned}
\end{equation}
Thus, if we choose $\Lambda$ such that $\Lambda^2=8C\|V\|_{L^{n/2}( \R^n)}C_{R_2}R_2\ell_2$, we obtain \eqref{3.43} for $\la>\Lambda$, which complete the proof of \eqref{3.19}.

The rest of the paper is devoted to the proof of Theorem~\ref{lemmm}, with a focus on the Fourier analysis related to the standard Laplacian $\Delta$ or free
Schr\"odinger equation in Euclidean space. In other words, we can assume $V\equiv 0$ from now on.

To proceed, if $\beta$ is as in \eqref{bump}, let us define
``wider cutoffs" that we shall also use as follows
\begin{equation}\label{l.1}
\tb(s)=\sum_{|j|<10}\beta(2^{-j}s)\in C^\infty_0((2^{-10},
2^{10})).
\end{equation}
For future use, note that
\begin{equation}\label{l.2}
\tb(s)=1 \quad \text{on } \, \, (1/4,4).
\end{equation}

For fixed $0<\e<1$, our first estimate concerns the operator 
\begin{equation}\label{l.3}
U(t)=\1_+(t) \tb(P/\la) e^{it\la^{-1}\Delta}e^{-\e t},
\end{equation}
where $\1_+(s)=\1_{[0,+\infty)}(s)$ denote the Heaviside function and $P=\sqrt{-\Delta}$.
For later use,
let us note that we can rewrite this operator.
Indeed,
if we recall that
\begin{equation*}
(2\pi)^{-1}\int_{-\infty}^\infty \frac{e^{it\tau}}{i\tau +\e} \, d\tau = \1_+(t)e^{-\e t},
\end{equation*}
we deduce that  
\begin{equation}\label{l.4}
U(t)f(x)=\frac{i\la}{2\pi}
\int_{-\infty}^\infty 
\frac{e^{it\tau}}{-\la\tau + \Delta+i\e\la}
\, \tb(P/\la)f(x) \, d\tau.
\end{equation}
Also, if we regard $U$ as an operator sending functions
of $x$ into functions of $x,t$, then its adjoint is the operator
\begin{equation}\label{l.5}
U^*F(x)=\int_0^\infty e^{-\e s}\bigl(e^{-is\la^{-1}\Delta}
\tb(P/\la) F(s,\cd )\bigr)(x) \, ds.
\end{equation}
Consequently,
\begin{multline}\label{l.6}
\int U(t)U^*(s)F(s,x)\, ds
\\
=\1_+(t) \int_0^\infty
\Bigl(e^{i(t-s)\la^{-1}\Delta}e^{-\e(t-s)}
\tb^2(P/\la) e^{-2\e s}F(s,\cd)\Bigr)(x) \, ds.
\end{multline}
Note also that if, say,
\begin{equation}\label{l.7}
F(t,x)=0, \quad t\notin [0,\e^{-1}],
\end{equation}
then
 the solution to the scaled
inhomogeneous Schr\"odinger equation with damping term
\begin{equation}\label{l.8}
(i\la \partial_t+\Delta_g+i\e\la)w(t,x)=F(t,x), \quad
w(0,\cd)=0
\end{equation}
is given by
\begin{multline}\label{l.9}
w(t,x)=(i\la)^{-1}\int_0^t 
 \bigl(\Bigl(e^{i(t-s)\la^{-1}\Delta}e^{-\e(t-s)}
F(s,x)\bigr) \, ds
\\
=(2\pi)^{-1}\int_0^{\e^{-1}}\int_{-\infty}^\infty
\frac{e^{i(t-s)\tau}}{-\la\tau+\Delta+i\e\la}
F(s,\cd)(x)\, d\tau ds.
\end{multline}
Thus, since $w(t,\cd)=0$ for $t<0$, it follows
from \eqref{l.4}, \eqref{l.5} and \eqref{l.9}
that  
\begin{multline}\label{l.10}
\tb^2(P/\la)w(t,x) =(i\la)^{-1}\int_0^t
U(t)U^*(s) \bigl(e^{2\e s}F(s,\cd)\bigr)(x) \, ds
\\
= (2\pi)^{-1}\int_0^{\e^{-1}}\int_{-\infty}^\infty
\frac{e^{i(t-s)\tau}}{-\la\tau+\Delta+i\e\la}
\tb^2(P/\la)F(s,\cd)(x)\, d\tau ds.
\end{multline}

Using these formulas, we claim that we can use the Keel-Tao \cite{KT} theorem to
deduce the following.

\begin{proposition}\label{mainprop}  Suppose that
$F$ satisfies the support assumption in \eqref{l.7}
and that $w$ solves \eqref{l.8}.  Then for $\la\ge1$ we have
\begin{equation}\label{l.11}
\bigl\|\tb^2(P/\la)w\bigr\|_{L^2_tL^{\frac{2n}{n-2}}_x(\R\times \R^n)}
\lesssim \la^{-1/2}\e^{-1/2}\|F\|_{L^2_{t,x}([0,\e^{-1}]\times \R^n)},
\end{equation}
and also
\begin{equation}\label{l.12}
\bigl\|\tb^2(P/\la)w\bigr\|_{L^2_tL^{\frac{2n}{n-2}}_x(\R\times \R^n)}
\lesssim 
\|F\|_{L^{2}_tL^{\frac{2n}{n+2}}_x([0,\e^{-1}]\times \R^n)}.
\end{equation}
Additionally, if for all times $t$, $\supp F(t, \cdot)\subset B_R$, with $B_R$ being a fixed ball of radius $R\ge1$ in $\R^n$ centered at origin,
we have 
\begin{equation}\label{l.13}
\bigl\|\tb^2(P/\la)w\bigr\|_{L^2_tL^{\frac{2n}{n-2}}_x(\R\times B_R)}
\le C_R \la^{-1/2}
\|F\|_{L^{2}_tL^{2}_x([0,\e^{-1}]\times \R^n)}.
\end{equation}
\end{proposition}

\begin{proof}
Let
\begin{equation}\label{l.14}
V(t')f(x)= U(\la t')f(x)=\1_+(t')e^{-\e\la t'} \tb(P/\la) e^{it'\Delta_g}f(x).
\end{equation}
We then clearly have
\begin{equation*}
\|V(t')\|_{L^2(\R^n)\to L^2(\R^n)}=O(1),
\end{equation*}
and since
\begin{multline*}
V(t')(V(s'))^*f (x)=\\ \1_+(t')\1_+(s')e^{-\e\la (t'+s')} (2\pi)^{-n}\int_{\R^n}\int_{\R^n} e^{i\langle x-y,\xi \rangle+(t'-s')|\xi|^2}\tb^2(|\xi|/\la) f(y) d\xi dy,
\end{multline*}
by stationary phase methods, we have
$$\|V(t')(V(s'))^*\|_{L^1(\R^n)\to L^\infty(\R^n)}.
\lesssim |t'-s'|^{-n/2}.$$

We can use the Keel-Tao theorem along with these
two inequalities to deduce 
that
$$\|V(t')f\|_{L^2_{t'}L^{{\frac{2n}{n-2}}}_x(\R\times \R^n)}
\lesssim \|f\|_{L^2(\R^n)},
$$
as well
as 
$$\Bigl\|\int_0^{t'}V(t')V^*(s')G(s',\cd)\, ds'
\, \Bigr\|_{L^2_{t'}L^{\frac{2n}{n-2}}_x(\R\times \R^n)}
\lesssim \|G\|_{L^{2}_{t}L^{{\frac{2n}{n+2}}}_x(\R\times
\R^n)},
$$
and
$$\Bigl\| \int_0^\infty V^*(s')G(s', \cd)\, ds'
\Bigr\|_{L^2(\R^n)}
\lesssim  \|G\|_{L^{2}_tL^{{\frac{2n}{n+2}}}_x(\R\times
\R^n)}.
$$
Using \eqref{l.14} we deduce that these inequalities
are equivalent to
\begin{equation}\label{l.15}
\|U(t)f\|_{L^2_tL^{\frac{2n}{n-2}}_x(\R\times \R^n)}
\lesssim \la^{1/2}\|f\|_{L^2(\R^n)},
\end{equation}
as well as
\begin{equation}\label{l.16}
\Bigl\|\int_0^t U(t)U^*(s)H(s,\cd)\, ds
\Bigr\|_{L^2_tL^{\frac{2n}{n-2}}_x(\R\times \R^n)}
\\
\lesssim \la \,
\|H\|_{L^{2}_tL^{\frac{2n}{n+2}}_x(\R\times
\R^n)}
\end{equation}
and
\begin{equation}\label{l.17}
\Bigl\|\int_0^\infty U^*(s)H(s,\cd)\, ds\Bigr\|_{L^2(\R^n)}
\lesssim \la^{1/2}\|H\|_{L^{2}_tL^{\frac{2n}{n+2}}_x(\R\times
\R^n)},
\end{equation}
respectively.

Using \eqref{l.10} with $H=e^{2\e s}F$ along with \eqref{l.16} we obtain \eqref{l.12}
since $\|H\|_{L^2_tL^{\frac{2n}{n-2}}_x}\approx \|F\|_{L^2_tL^{\frac{2n}{n-2}}_x}$
due to \eqref{l.7}. Moreover, since $\|U^*(s)\|_{L^2(\R^n)\to L^2(\R^n)}=O(1)$,
using \eqref{l.10} along with \eqref{l.15}
we find that if $H=e^{2\e s}F$
\begin{align*}
\bigl\|\tb^2(P/\la)w\bigr\|_{L^2_tL^{\frac{2n}{n-2}}_x(\R\times \R^n)}
&\le \la^{-1} \int_0^{\e^{-1}}
\bigl\|\1_+(t-s) U(t)U^*(s)H(s, \cd)
\bigr\|_{L^2_tL^{\frac{2n}{n-2}}_x(\R\times \R^n)}\, ds
\\
&\lesssim
\la^{-1/2}\int_0^{\e^{-1}}\|U^*(s)H(s, \cd)\|_{L^2_x}\, ds
\\
&\lesssim \la^{-1/2}\int_0^{\e^{-1}}\|F(s,\cd)\|_{L^2_x}\, ds
\le \la^{-1/2}\e^{-1/2}\|F\|_{L^2_{t,x}([0,\e^{-1}]\times \R^n)},
\end{align*}
as desired, which is \eqref{l.11}.

To prove \eqref{l.14}, note that if we assume
\begin{equation}\label{l.18}
\tb^2(P/\la)w(t,x)=(i\la)^{-1}\int_{\R^n}\int_{\R} K(t,s,x,y) F(s,y)ds dy,
\end{equation}
then by \eqref{l.10}, we have
\begin{equation}\label{l.19}
K(t,s,x,y)=(2\pi)^{-n}\1_{[0,t]}(s)e^{-\e(t-s)}\int_{\R^n} e^{i\langle x-y,\xi \rangle+(t-s)\la^{-1}|\xi|^2}\tb^2(|\xi|/\la) d\xi.
\end{equation}
Since, in proving \eqref{l.13}, we only consider the case where $x,y \in B_R$, by the definition of $\tb$ in \eqref{l.1} and using integration by parts, 
\begin{equation}\label{error1}
|K(t,s,x,y)|\le C_N \la^n (1+\la|x-y|+\la|t-s|)^{-N},  \,\,\,\text{if}\,\,\, |t-s|>2^{11}R.
\end{equation}

Now let us choose a function $\eta\in C_0^\infty (\R)$ satisfying $\eta(t)=0$, if $|t|>1$, and $\sum_{j=-\infty}^\infty \eta(t-j)\equiv 1$. Given $R\ge1$,
we shall set
\begin{equation}\label{eta}
\eta_j(t)=\eta_{j,R}(t)=\eta((2^{11}R)^{-1}\cdot t-j),
\end{equation}
and write
\begin{multline}
K(t,s,x,y)=K_1(t,s,x,y)+K_2(t,s,x,y) \\
=\sum_{j,s \in \mathbb{Z}:|j-k|\le 1} \eta_j(t)\eta_k(s)K(t,s,x,y)+\sum_{j,s \in \mathbb{Z}:|j-k|> 1} \eta_j(t)\eta_k(s)K(t,s,x,y).
\end{multline}

As a consequence of \eqref{error1} and \eqref{eta}, we have
\begin{equation}\label{error2}
|K_2(t,s,x,y)|\le C_N \la^n \la^{-N}R^{-N}(1+\la|x-y|+\la|t-s|)^{-N},
\end{equation}
which, after an application of Young's inequality, gives us much better bounds than the right side of \eqref{l.13}.

Hence, it suffices to consider the case where $t, s$ are supported near diagonal, i.e., we need to show that
\begin{equation}\label{l.20}
\bigl\|K_1F\bigr\|_{L^2_tL^{\frac{2n}{n-2}}_x(\R\times B_R)}
\le C_R \la^{1/2}
\|F\|_{L^{2}_tL^{2}_x([0,\e^{-1}]\times \R^n)}.
\end{equation}

Firstly, for fixed $j, k$, using \eqref{l.10} along with \eqref{l.15}, we find that if $H=\eta_k(s)e^{2\e s}F$,
\begin{align*}
\bigl\|\iint \eta_j(t)\eta_k(s)&K(t,s,x,y)F(s,y)dsdy\bigr\|_{L^2_tL^{\frac{2n}{n-2}}_x(\R\times \R^n)} \\
&\le  \int_0^{\e^{-1}}
\bigl\|\1_+(t-s) U(t)U^*(s)H(s, \cd)
\bigr\|_{L^2_tL^{\frac{2n}{n-2}}_x(\R\times \R^n)}\, ds
\\
&\lesssim
\la^{1/2}\int_0^{\e^{-1}}\|U^*(s)H(s, \cd)\|_{L^2_x}\, ds
\\
&\lesssim \la^{1/2}\int_0^{\e^{-1}}\|H(s,\cd)\|_{L^2_x}\, ds
\le \la^{1/2}R^{1/2}\|H\|_{L^2_{t,x}([0,\e^{-1}]\times \R^n)},
\end{align*}
where we used the fact that $H(s,\cdot)$ is supported in a interval of length $\approx 2^{11}R$ in the last inequality.

Since $\eta_j(t)\bar\eta_k(t)=0$ if $|j-k|>1$, by Cauchy Schwartz and the above inequality,
\begin{align*}
\bigl\|\iint K_1(t,s,x, &y)F(s,y)dsdy\bigr\|^2_{L^2_tL^{\frac{2n}{n-2}}_x(\R\times \R^n)} \\
&\le \sum_j \bigl\|\sum_{k:|k-j|\le1}\iint \eta_j(t)\eta_k(s)K(t,s,x,y)F(s,y)dsdy\bigr\|^2_{L^2_tL^{\frac{2n}{n-2}}_x(\R\times \R^n)} \\
&\le \sum_j \sum_{k:|k-j|\le1}\bigl\|\iint \eta_j(t)\eta_k(s)K(t,s,x,y)F(s,y)dsdy\bigr\|^2_{L^2_tL^{\frac{2n}{n-2}}_x(\R\times \R^n)} \\
&\le \sum_j \sum_{k:|k-j|\le1}\la R\|\eta_k(t)e^{2\e t}F(x)\|^2_{L^2_{t,x}([0,\e^{-1}]\times \R^n)} \le \la R \|F\|_{L^{2}_tL^{2}_x([0,\e^{-1}]\times \R^n)},
\end{align*}
which completes the proof of \eqref{l.20}.
\end{proof}

Finally, to prove Theorem~\ref{lemmm} using Proposition~\ref{mainprop}, we shall basically repeat the argument at the end of Section 2 in \cite{huang2020quasimode}.
To proceed, we shall require the following two elementary lemmas which are analogous to Lemma 2.2 and Lemma 2.3 in \cite{huang2020quasimode}.
\begin{lemma}\label{soblemma}  Let $\alpha\in C([0,\infty))$ and $1<p\le 2<q<\infty$.  Then
\begin{equation}\label{m.28}
\|\alpha(P)f\|_{L^q(\R^n)}
\le C_{p,q} \, 
\bigl(\sup_{\mu\ge 0} (1+\mu)^{n(\frac1p-\frac1q)}
|\alpha(\mu)|\bigr) \, \|f\|_{L^p(\R^n)}.
\end{equation}
\end{lemma}

\begin{lemma}\label{freeze}  Suppose that
\begin{equation}\label{m.29}
|K_\la(t,t')|\le \la (1+\la|t-t'|)^{-2}.
\end{equation}
Then if $1\le p\le q\le \infty$ we have
the following uniform bounds for $\la\ge1$
\begin{equation}\label{m.30}
\Bigl\|\, \int K_\la(t,t') \, 
G(t',\cd)\, dt' \, 
\Bigr\|_{L^p_tL^q_x(\R\times \R^n)}
\le C\|G\|_{L^p_tL^q_x(\R\times \R^n)}.
\end{equation}
Also, suppose that 
$$WF(t,x)=
\int_{-\infty}^\infty \int_{\R^n}
K(t,x;t',y) \, F(t',y) \, dy \, dt'
$$
and that for each $t,t'\in \R$ the operator
$$W_{t,t'}f(x) = \int_{\R^n} K(x,t;y,t') f(y) \, dy$$
satisfies
$$\|W_{t,t'}f\|_{L^q(\R^n)}
\le \la(1+\la|t-t'|)^{-2} \, \|f\|_{L^r(\R^n)}
$$
for some $1\le r\le q\le \infty$.  Then if
$1\le s\le p\le \infty$ we have for $\la\ge1$
\begin{equation}\label{m.31}
\|WF\|_{L^p_tL^q_x(\R\times \R^n)}
\le C\la^{\frac1s-\frac1p}
\|F\|_{L^s_tL^r_x(\R\times \R^n)}.
\end{equation}
\end{lemma}
Both lemmas are well known,
The first lemma is a direct consequence of Sobolev estimates and spectral theorem, while the second lemma is essentially Theorem 0.3.6 in \cite{SFIO2}. For more 
details about the proof of the two lemmas, see, i.e., \cite{huang2020quasimode}.

\begin{proof}[Proof of Theorem~\ref{lemmm}]
We first note that the kernel of $\beta(-D_t/\la)$
is $O(\la (1+\la|t-t'|)^{-2})$.  Therefore, by
\eqref{m.30}
$$\|\beta(-D_t/\la) \tb^2(P/\la)w\|_{L^p_tL^q_x(\R\times \R^n)}
\lesssim \| \tb^2(P/\la)w\|_{L^p_tL^q_x(\R\times \R^n)}.
$$
Therefore, if as in Proposition~\ref{mainprop} and
our theorem our forcing term $F$ satisfies \eqref{i.12},
it suffices to show that $\beta(-D_t/\la)(I-\tb^2(P/\la))w$ enjoys the bounds in 
\eqref{i.13}, \eqref{i.14} and \eqref{i.15}.

Recalling \eqref{l.9}, this means that it suffices
to show that
\begin{multline}\label{m.32}
\Bigl\|
\int_0^{\e^{-1}}\int_{-\infty}^\infty
\frac{e^{i(t-s)\tau}}{-\la\tau-P^2+i\e\la}
\beta(-\tau/\la) \, 
\bigl(1-\tb^2(P/\la)\bigr)\, F(s,\cd)\, d\tau ds
\, \Bigr\|_{L^2_tL^{\frac{2n}{n-2}}_x(\R\times \R^n)}
\\
\lesssim \la^{-1/2}\|F\|_{L^2_{t,x}(\R\times \R^n)},
\end{multline}
as well as
\begin{multline}\label{m.33}
\Bigl\|
\int_0^{\e^{-1}}\int_{-\infty}^\infty
\frac{e^{i(t-s)\tau}}{-\la\tau-P^2+i\e\la}
\beta(-\tau/\la) \, 
\bigl(1-\tb^2(P/\la)\bigr)\, F(s,\cd)\, d\tau ds
\, \Bigr\|_{L^2_tL^{\frac{2n}{n-2}}_x(\R\times \R^n)}
\\
\lesssim 
\|F\|_{L^{2}_tL^{\frac{2n}{n+2}}_x(\R\times \R^n)}.
\end{multline}
Actually, \eqref{m.32} also implies that $\beta(-D_t/\la)(I-\tb^2(P/\la))w$ enjoys a better bound than $\tb^2(P/\la)w$ since we do not have the $\e^{-1/2}$ factor on 
the right side of \eqref{i.13}.

To use Lemma~\ref{freeze} set
$$\alpha(t,s;\mu)
= \int_{-\infty}^\infty
\frac{e^{i(t-s)\tau}}{-\la\tau-\mu^2+i\e\la}
\beta(-\tau/\la) \, 
\bigl(1-\tb^2(\mu/\la)\bigr) \, d\tau,$$
and note that, by \eqref{l.2} and the support
properties of $\beta$ we have
for $j=0,1,2$
$$
\la \,
\bigl|
\la^{j} \,
\partial_\tau^j \bigl((1-\tb^2(\mu/\la)\bigr)  \beta(-\tau/\la)
(-\la\tau-\mu^2+i\e\la)^{-1}\bigr)\, \bigr|
\lesssim \la (\mu^2+\la^2)^{-1},$$
which, by a simple  integration parts argument, translates to the bound
$$|\alpha(t,s;\mu)|\lesssim 
\la(1+\la|t-s|)^{-2} \cdot (\mu^2+\la^2)^{-1}.$$

If we use Lemma~\ref{soblemma} we deduce from this that the 
``frozen operators" 
$$T_{t,s}h(x)
=\int_{-\infty}^\infty
\frac{e^{i(t-s)\tau}}{-\la\tau-P^2+i\e\la}
\beta(-\tau/\la) \, 
\bigl(1-\tb^2(P/\la)\bigr)h(x)\, d\tau,$$
satisfy
\begin{equation}\label{m.34}
\|T_{t,s}h\|_{L^{\frac{2n}{n-2}}(\R^n)}
\lesssim \la(1+\la|t-s|)^{-2}
\cdot \la^{-2+1}\|h\|_{L^{2}(\R^n)}
\end{equation}
as well as
\begin{equation}\label{m.35}
\|T_{t,s}h\|_{L^{\frac{2n}{n-2}}(\R^n)}
\lesssim \la(1+\la|t-s|)^{-2}\|h\|_{L^{\frac{2n}{n+2}}(\R^n)}
\end{equation}
due to the fact that $n(\frac12-\frac{n-2}{2n})=1$
 and $n(\frac{n+2}{2n}-\frac{n-2}{2n})=2$.

If we combine \eqref{m.34} and \eqref{m.31}, we conclude
that the left side of \eqref{m.32} is dominated by
$\la^{-1} \|F\|_{L^2_{t,x}(\R\times \R^n)},
$
which is better than the bounds posited in \eqref{m.32}
by a factor of $\la^{-1/2}$.
Similarly, if we combine \eqref{m.35} and \eqref{m.31}, we find that the left side of \eqref{m.33} is dominated
by
$ \|F\|_{L^{2}_t
L^{\frac{2n}{n+2}}_x(\R\times \R^n)},
$ which completes
the proof.
\end{proof}

%\setcitestyle{numbers} % set the citation style to ``numbers''.
\bibliography{02-10}
\bibliographystyle{abbrv}
\end{document}